\documentclass[12pt]{article}

\usepackage{graphicx,xypic,pstricks}
\usepackage{amsthm,amsmath,amssymb,amsfonts,latexsym}
\input xy
\xyoption{all}

\theoremstyle{definition} 

 \newtheorem{definition}{Definition}[section]

\theoremstyle{plain}      

 \newtheorem{proposition}[definition]{Proposition}
 \newtheorem{theorem}[definition]{Theorem}
 
 \newtheorem{lemma}[definition]{Lemma}

\newcommand{\R}{\mathbb{R}}
\newcommand{\Z}{\mathbb{Z}}
\newcommand{\C}{\mathbb{C}}

\newcommand{\g}{\mathcal{G}}
\newcommand{\boM}{\mathcal{M}}
\newcommand{\boF}{\mathcal{F}}
\newcommand{\boR}{\mathcal{R}}

\newcommand{\boL}{\mathcal{L}}
\newcommand{\boH}{\mathcal{H}}
\newcommand{\hol}{\textrm{Hol}}

\newcommand{\Ad}{\textrm{Ad}}
\newcommand{\ang}{\textrm{ang}}
\newcommand{\reg}{\textrm{reg}}
\newcommand{\bi}{{\bf i}}
\newcommand{\bj}{{\bf j}}
\newcommand{\bk}{{\bf k}}
\newcommand{\tgl}{\tilde{\textrm{GL}}}

\newcommand{\tr}{\textrm{Tr}}
\newcommand{\dd}{\textrm{d}}
\newcommand{\id}{\textrm{Id}}
\newcommand{\inter}{\textrm{Int}}
\newcommand{\vol}{\textrm{Vol}}
\renewcommand{\sp}{\textrm{Sp}}

\DeclareMathOperator{\coker}{coker}
\DeclareMathOperator{\rk}{rk}
\DeclareMathOperator{\im}{im}
\DeclareMathOperator{\Hom}{Hom}
\renewcommand{\phi}{\varphi}
\renewcommand{\epsilon}{\varepsilon}
\newcommand{\su}{\textrm{SU}(2)}
\newcommand{\so}{\textrm{SO}(3)}
\newcommand{\lu}{\textrm{su}(2)}
\newcommand{\tl}[1]{\widetilde{#1}}
\newcommand{\ba}[1]{\overline{#1}}

\newcommand{\te}{\tilde{e}}

\title{Geometry of representation spaces in $\su$}

\author{Julien March\'e\footnote{
Centre de Math\'ematiques Laurent Schwartz,\'Ecole Polytechnique,
Route de Saclay, 91128 Palaiseau Cedex, France,
email:\,\tt{marche@math.polytechnique.fr}
}}
\date{}
\begin{document}
\maketitle

\begin{abstract}
These notes of a course given at IRMA in April 2009 cover some aspects of the representation theory of fundamental groups of manifolds of dimension at most 3 in compact Lie groups, mainly $\su$. We give detailed examples, develop the techniques of twisted cohomology and gauge theory. We review Chern-Simons theory and describe an integrable system for the representation space of a surface. Finally, we explain some basic ideas on geometric quantization. We apply them to the case of representation spaces by computing Bohr-Sommerfeld orbits with metaplectic correction.
\end{abstract}



\tableofcontents

\section{Introduction}

Representations of fundamental groups of manifolds of dimension 2 and 3 in a compact Lie group have a long history. In the case of surfaces, they appeared after the development of Teichmuller theory, for instance to classify holomorphic vector bundles. In the case of 3-manifolds, they were used to help distinguishing 3-manifolds, as knot complements or for geometrization purposes.
Then, in the eighties, representations of surfaces were much studied for their symplectic properties until Witten discovered a  deep relationship between these representation spaces and the Jones polynomial of knots, via Chern-Simons quantum field theory. This field is still very lively and these notes were written as a preparation for understanding this relationship. We planned to give a quick review of the geometric aspects of the representation spaces.

In a first part, we give some examples of representation spaces for surfaces and knots. They will help the reader to understand the second part where we introduce the basic tool in order to understand the differential geometry of representation spaces : twisted (co)homology. We give a brief account on the symplectic structure on surfaces and on Reidemeister torsion.
We introduce gauge theory in the third part and review the symplectic structure of surface representations in this context. This third part is mostly an introduction to the fourth where we explain the basic constructions of Chern-Simons theory and its applications to the geometry of representation spaces.
The fifth part studies in more details the representation space of a closed surface of genus at least 2 by introducing trace functions and a related integrable system.
In the last part, we give an introduction to geometric quantization, insisting on Lagrangian fibrations and spin structures. With the help of some examples, we treat the case of representation space of surfaces.

\section{Examples of representation spaces}

\subsection{Generalities on SU(2)}
We define the group $\su$ as the group of matrices $M\in \text{M}(2,\C)$ satisfying $M\ba{M}^T=1$ and $\det M=1$.
We can write this set alternatively as
$$\su=\{
\begin{pmatrix}\alpha&-\ba{\beta}\\ \beta&\ba{\alpha}\end{pmatrix}
,\alpha,\beta \in \C, |\alpha|^2+|\beta|^2=1\}.$$
This last description shows that $\su$ is topologically a sphere $S^3$.

We often look at $\su$ as the unit sphere in the space of quaternions $\mathbb{H}$ where the standard generators are the following
$$\bf{i}=\begin{pmatrix}i&0\\0&-i\end{pmatrix}, \bf{j}=\begin{pmatrix}0&1\\-1&0\end{pmatrix},\bf{k}=\begin{pmatrix}0&i\\i&0\end{pmatrix}.$$
These generators also form a basis of the Lie algebra $\lu$ of $\su$.
Observe the following elementary fact: given $M\in \su$, there is a unique $\phi\in [0,\pi]$ such that $M$ is conjugate to $\begin{pmatrix}e^{i\phi}&0\\0&e^{-i\phi}\end{pmatrix}$. We will call $\phi$ the angle of $M$, it can be easily computed via the formula $\tr(M)=2\cos(\phi)$. We will define $\ang(M)=\arccos(\tr(M)/2)$.

\subsection{Generalities on representation spaces}
Let $\Gamma$ be a finitely generated group. For our purposes, $\Gamma$ will be the fundamental group of a compact manifold or of a finite CW-complex. We denote by $\boR(\Gamma,\su)$ the set of homomorphisms from $\Gamma$ to $\su$. 

As $\Gamma$ is finitely generated, say by $t_1,\ldots,t_n$, any element $\rho\in \boR(\Gamma,\su)$ is determined by the image of the generators. This gives an embedding of $\boR(\Gamma,\su)$ into $\su^n$ sending $\rho$ to the family $(A_i=\rho(t_i))_{i\le n}$. Moreover, let $R_1,\ldots,R_m$ be a generating family of relations for $\Gamma$. A family $(A_i)$ defines a representation if and only if $R_j(A_i)=1$ for all $j\in \{1\ldots m\}$.
Here are two easy consequences:
\begin{enumerate}
\item $\boR(\Gamma,\su)$ is a topological space (as a subspace of $\su^n$).
\item $\boR(\Gamma,\su)$ is a real algebraic variety as $\su$ is algebraic and relations are algebraic maps. More precisely, $\su$ may be defined as $\{(x,y,z,t)\in \R^4, x^2+y^2+z^2+t^2=1\}$ putting $\alpha=x+iy$ and $\beta=z+it$.
\end{enumerate}
It is easy to verify that these structures do not depend on the choices of generators and relations.

We will say that two representations $\rho,\rho'$ are conjugate if there is $M\in \su$ such that $\rho'=M\rho M^{-1}$. We denote by $\boM(\Gamma,\su)$ the set of conjugacy classes of representations, or the quotient $\boR(\Gamma,\su)/\su$.
This construction allows us to define $\boM(X,\su)=\boM(\pi_1(X),\su)$ for any topological space $X$. The ambiguity in $\pi_1(X)$ is precisely described by a conjugation and hence disappear in the quotient. We deduce the following consequences for $\boM(X,\su)$ where $\pi_1(M)$ is finitely generated:
\begin{enumerate}
\item $\boM(\Gamma,\su)$ is a topological space (quotient topology).
\item $\boM(\Gamma,\su)$ can be given a structure of a real algebraic variety, but we will not deal with this topic in these notes.
\end{enumerate}
\subsection{Basic examples}
Let us remove $\su$ from our notation. Our first example is  $\boM(S^1)$ which is the set of conjugacy classes of $\su$. It is homeomorphic to $[0,\pi]$ via the angle map. The two boundary points correspond to the conjugacy classes of the central elements $\pm 1$.

Let $X=S^1\vee S^1$. Then $\boM(X)=\{(A,B)\in \su^2\}/\su$ is the set of conjugacy classes of pairs of matrices. 
Let $\phi$ and $\psi$ be the angles of $A$ and $B$ respectively. One can suppose up to conjugation that $A=\begin{pmatrix}e^{i\phi}&0\\ 0&e^{-i\phi}\end{pmatrix}$ and that there exists $P\in \su$ such that $B=P\begin{pmatrix}e^{i\psi}&0\\ 0&e^{-i\psi}\end{pmatrix}P^{-1}$. Multiplying $P$ on the right by $\begin{pmatrix}e^{iy}&0\\ 0&e^{-iy}\end{pmatrix}$ do not change $B$ whereas multiplying $P$ on the left by $\begin{pmatrix}e^{ix}&0\\ 0&e^{-ix}\end{pmatrix}$ conjugates $B$ by a matrix which do not act on $A$ by conjugation.
Given $P=\begin{pmatrix}\alpha&-\ba{\beta}\\ \beta&\ba{\alpha}\end{pmatrix}$, we compute 
$\begin{pmatrix}e^{ix}&0\\0&e^{-ix}\end{pmatrix} P \begin{pmatrix}e^{iy}&0\\ 0&e^{-iy}\end{pmatrix}=
\begin{pmatrix}\alpha e^{i(x+y)} & -\ba{\beta}e^{i(x-y)}\\ \beta e^{i(y-x)} & \ba{\alpha}e^{-i(x+y)}\end{pmatrix}$. 
By setting $x+y=-\arg(\alpha), x-y=\arg(\beta)$ one can suppose that $\alpha$ and $\beta$ are real and non negative.
We compute $B=\begin{pmatrix} \alpha^2 e^{i\psi}+\beta^2e^{-i\psi} & \alpha\beta (e^{i\psi}-e^{-i\psi})\\
 \alpha\beta (e^{i\psi}-e^{-i\psi})& \beta^2e^{i\psi}+\alpha^2 e^{-i\psi}\end{pmatrix}$.
 This formula describes all possible values of $B$ where $\alpha,\beta\ge 0$ and $\alpha^2+\beta^2=1$. To show that all these pairs $(A,B)$ are not conjugate, we compute $\tr(AB)=e^{i\phi}(\alpha^2  e^{i\psi}+\beta^2e^{-i\psi})+e^{-i\phi}(\beta^2e^{i\psi}+\alpha^2 e^{-i\psi})=2\alpha^2\cos(\phi+\psi)+2\beta^2\cos(\phi-\psi)$.

 Let $\eta\in [0,\pi]$ be the angle of $AB$. We see that $\cos(\eta)$ is a convex combination of $\cos(\phi+\psi)$ and $\cos(\phi-\psi)$.
 We deduce that $\eta$ belongs to the interval $[|\phi-\psi|,\min(\phi+\psi,2\pi-\phi-\psi)]$.

 We can sum up our computations in the following proposition:
 \begin{proposition} \label{pants}
The map from $ \boM(S^1\vee S^1)$ to $\{(\phi,\psi,\eta)\in [0,\pi]^3, \phi+\psi+\eta\le 2\pi,\phi\le \psi+\eta, \psi\le\phi+\eta,\eta\le\phi+\psi\}$ sending
$[A,B]$ to the triple $(\ang(A),\ang(B),\ang(AB))$ is an homeomorphism.
  \end{proposition}
 
 The symmetry between $\phi,\psi,\eta$ can be explained by replacing $X$ with a pair of pants $\Sigma$, that is a disk with two holes. The three angles $\phi,\psi,\eta$ are the angles of the three boundary components, and the inequations they satisfy are symmetric with respect to these coordinates. In Figure \ref{tetra} is represented the moduli space $\boM(S^1\vee S^1)$ where axes correspond to the angles of $A,B$ and $AB$. Notice that the corners of this tetrahedron correspond to central representations whereas its boundary corresponds to abelian representations.

 \begin{figure}[h]\label{tetra}
 \begin{center}
 \includegraphics[width=7cm]{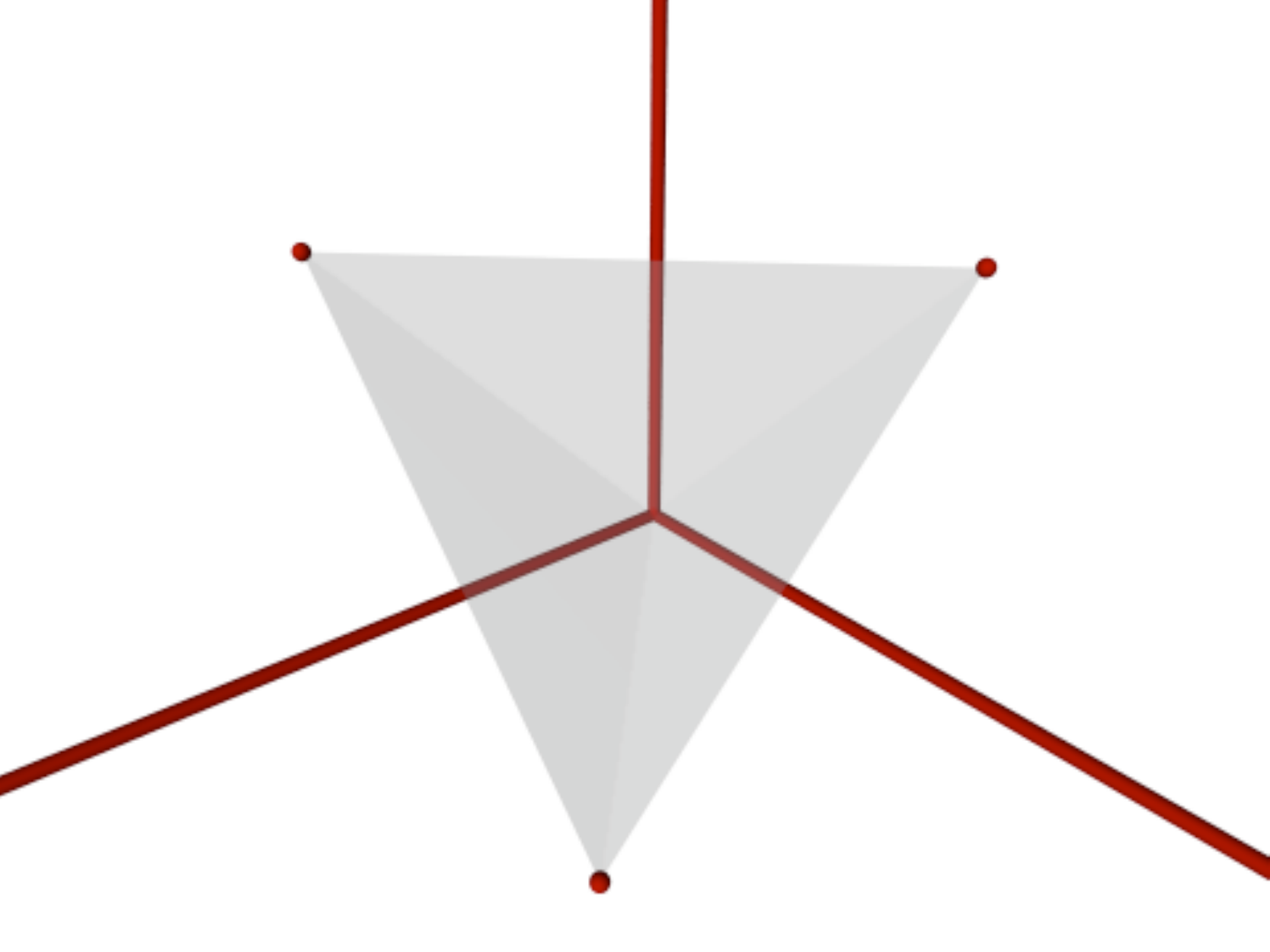}
 \caption{The moduli space of a free group with two generators}
 \end{center}
 \end{figure}

Our last easy but important example is the torus $S^1\times S^1$. One sees that $\boM(S^1\times S^1)=\{(A,B)\in \su^2, AB=BA\}/\su$.

The map sending $(\phi,\psi)\in\R/2\pi\Z\times \R/2\pi\Z$ to the representation defined by $A=\begin{pmatrix}e^{i\phi}&0\\ 0&e^{-i\phi}\end{pmatrix}$ and $\psi=\begin{pmatrix}e^{i\psi}&0\\ 0&e^{-i\psi}\end{pmatrix}$ is surjective as two commuting matrices are simultaneously diagonalizable. Moreover, if two pairs $(\phi,\psi)$ and $(\phi',\psi')$ are conjugate then either $(\phi',\psi')=(\phi,\psi)$ or $(\phi',\psi')=(-\phi,-\psi)$. One deduces that $\boM(S^1\times S^1)$ is homeomorphic to the torus $(\R/2\pi\Z)^2$ quotiented by the involution $(\phi,\psi)\mapsto(-\phi,-\psi)$. This is a sphere with 4 conical points of angle $\pi$ corresponding to representations with values in $\{\pm 1\}$. This moduli space appears as the boundary of the tetrahedron in Figure \ref{tetra}.

 With the same proof, we show that $\boM(S^1\times S^1\times S^1)$ is homeomorphic to $(\R/2\pi\Z)^3/I$ where $I(\phi,\psi,\eta)=(-\phi,-\psi,-\eta)$.

\subsection{Some representations of knot complements}\label{noeuds}
Let us look at two families of knots whose fundamental groups have the property of being presented by two generators and one relation, which makes the study of their representation space much simpler than the general case.

Before going further, let us introduce some terminology.
\begin{definition} Let $X$ be a topological space and $\rho\in \boR(X)$ a representation.
\begin{enumerate}
\item If $\rho$ takes values in $\{\pm 1\}$, we will say that $\rho$ is {\it central}.
\item If the image of $\rho$ is contained in an abelian subgroup of $\su$, we will say that $\rho$ is {\it abelian}.
\item In the other cases, $\rho$ will be said {\it irreducible}.
\end{enumerate}
\end{definition}
These definitions are invariant by conjugation hence we will use the same terminology for elements of $\boM(X)$.
A central representation is given by an homomorphism from $\pi_1(X)$ to $\Z/2$. These representations are in bijection with $H^1(X,\Z/2)$.
An abelian representation is given by an homomorphism from $\pi_1(X)$ to $S^1$, hence these representations are in bijection with $H^1(X,S^1)$.
However, as in the case of the torus, the inversion map $I:S^1\to S^1$ induces a map $I_*$ on $H^1(X,S^1)$ and conjugacy classes of abelian representations are in bijective correspondence with $H^1(X,S^1)/I_*$.

\subsubsection{Torus knots}

Let $F(\phi,\psi)=\left((2+\cos\phi)\cos\psi,(2+\cos\phi)\sin\psi,\sin\phi\right)$ be the standard embedding of $(\R/2\pi\Z)^2$ in $\R^3$.

Given $a,b$ two positive and relatively prime integers, we define the torus knot $T(a,b)$ as the image of the embedding $t\mapsto F(at,bt)$. One can see on Figure \ref{tore} the example of $T(5,2)$.

\begin{figure}[h]\label{tore}
\begin{center}
\includegraphics{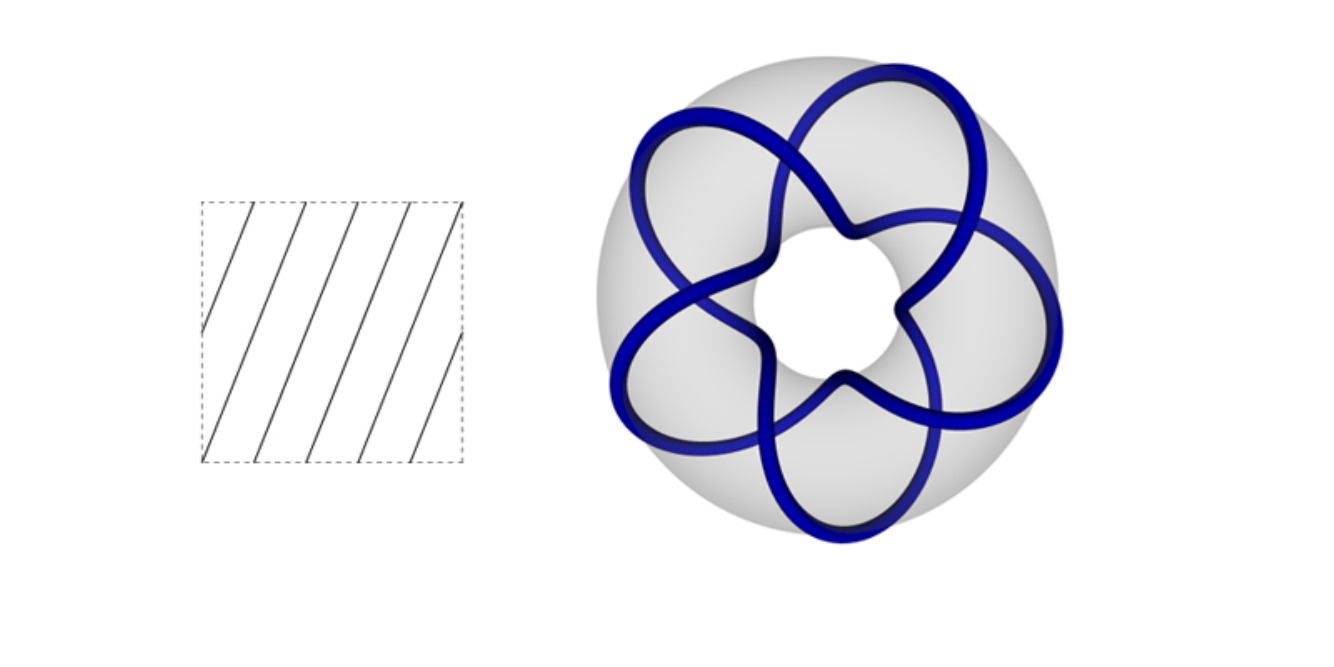}
\vspace{-1cm}
\caption{A torus knot}
\end{center}
\end{figure}
Let us look at this knot in $S^3=\R^3\cup \{\infty\}$. The complement of $T(a,b)$ cut along the torus supporting the knot has two components which are solid tori. Their intersection $C$ is the complement of the knot inside the torus which is an annulus, hence connected.
Then the Van-Kampen theorem asserts that $G_{a,b}=\pi_1(S^3\setminus T(a,b))=\langle u,v|u^a=v^b\rangle$. One can show indeed that the generator of $\pi_1(C)$ has order $a$ in one side and $b$ in the other, which gives the expression of $G_{a,b}$.

Let $\rho: G_{a,b}\to \su$ an irreducible representation. Then $\rho(u^a)=\rho(v^b)$ commutes with the image of $\rho$ and hence has to belong to the center. We deduce that $\rho(u)^{2a}=\rho(v)^{2b}=1.$

Then, the angle of $\rho(u)$ takes the values $k\pi/a$ for $0<k<a$ and the angle of $\rho(v)$ take the values $l\pi/b$ for $0<l<b$. As $\rho(u^a)=(-1)^k=\rho(v^b)=(-1)^l$ one has $k=l \mod 2$. The angle of $\rho(uv)$ determines the representation and takes its values in the non trivial interval $\pi[|k/a-l/b|,\min(k/a+l/b, 2-k/a-l/b)]$.
Hence the irreducible part of $\boM(G_{a,b})$ is a disjoint union of $(a-1)(b-1)/2$ arcs. The ends of these arcs are made of abelian representations which we describe know. To find abelian representations, one can suppose that $\rho(u)$ and $\rho(v)$ are diagonal with angle $\phi$ and $\psi$ respectively. Then the angles should satisfy $a\phi=b\psi\mod 2\pi$. We obtain all solutions by taking $\phi=bt, \psi=at$ and letting $t$ in $[0,\pi]$. The ends of the irreducible segments are equally reparted on the reducible segment as in Figure \ref{trefle} where we see on the left the abstract moduli space of the trefoil knot $T(3,2)$ and on the right, the way it is embedded in the tetrahedron.

\begin{figure}[h]\label{trefle}
\begin{center}
\includegraphics[width=12cm]{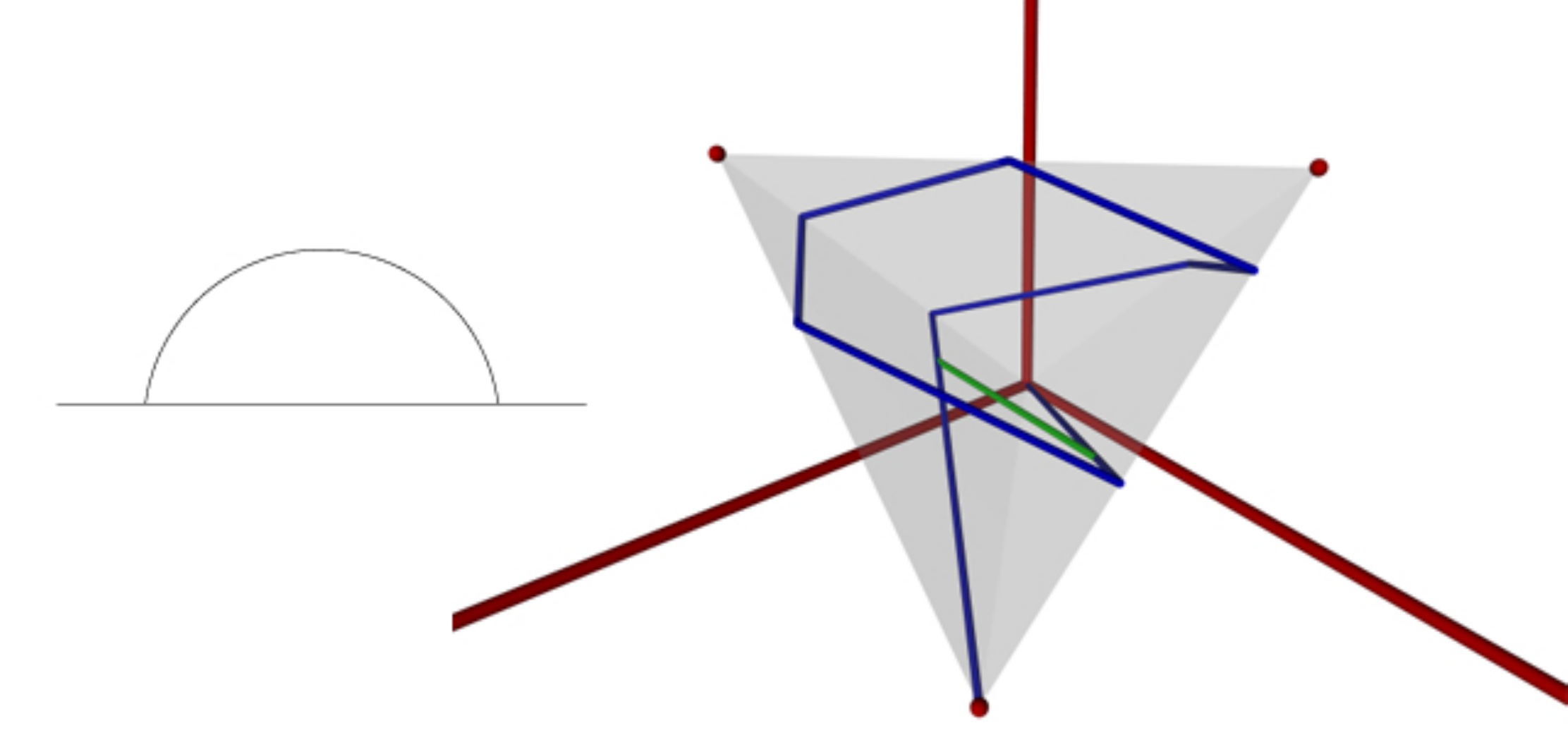}
\caption{Moduli space of the trefoil}
\end{center}
\end{figure}

\subsubsection{Two-bridge knots}
A two-bridge knot is a knot in $\R^3$ which is in Morse position relatively to some coordinate and has only 2 maxima and minima.

All 2-bridge knots can be put in a standard projection called Schubert normal form, see \cite{bz,sch}. Let $a$ and $b$ be relative integers such that $a$ is positive, $b$ is odd and the inequality $-a<b<a$ holds.
Formally, we define the projection of the two bridge knot $B(a,b)$ as the unique diagram obtained by gluing two copies of the disc in the left hand side of figure \ref{twobridge} by a diffeomorphism of the boundary circle sending $p_k$ to $p_{b-k}$ for $k\in \Z/2a\Z$.
We draw on the right of the figure the example of the knot $B(5,3)$.

\begin{figure}[htbp]\label{twobridge}
\begin{center}
\begin{pspicture}(-0.5,0)(7,7)
\includegraphics[height=6cm,width=11cm]{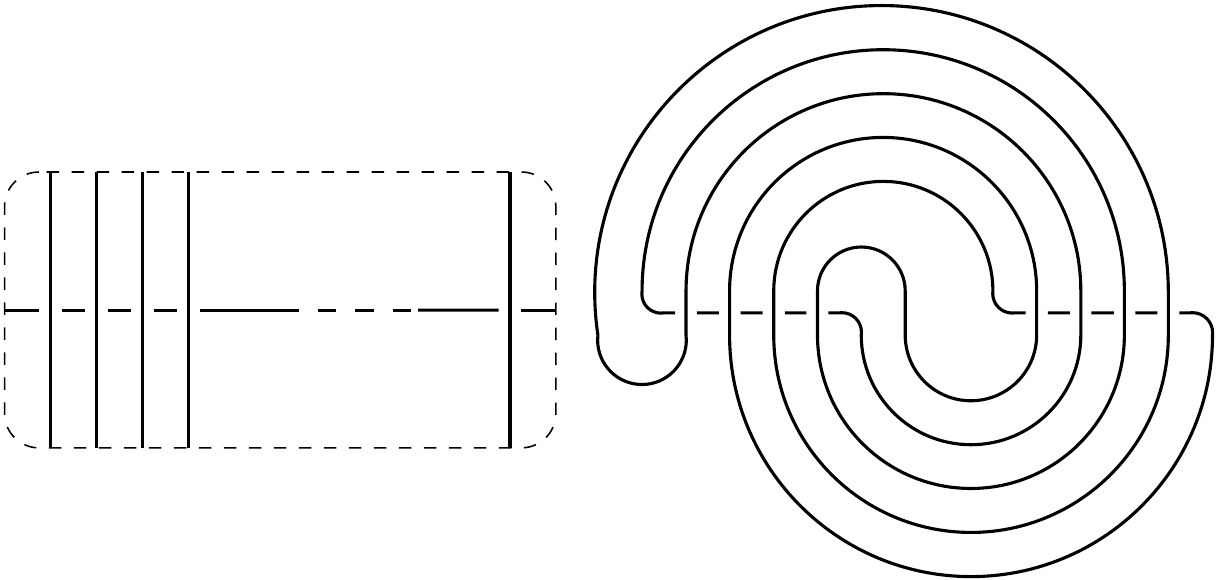}
\put(-11.4,2.8){$p_0$}
\put(-11,4.4){$p_1\cdots$}
\put(-6.5,4.4){$p_{b-1}$}
\put(-11,1){$p_{2b-1}\cdots$}
\end{pspicture}
\vspace{-1cm}
\caption{A two-bridge knot}
\end{center}
\end{figure}
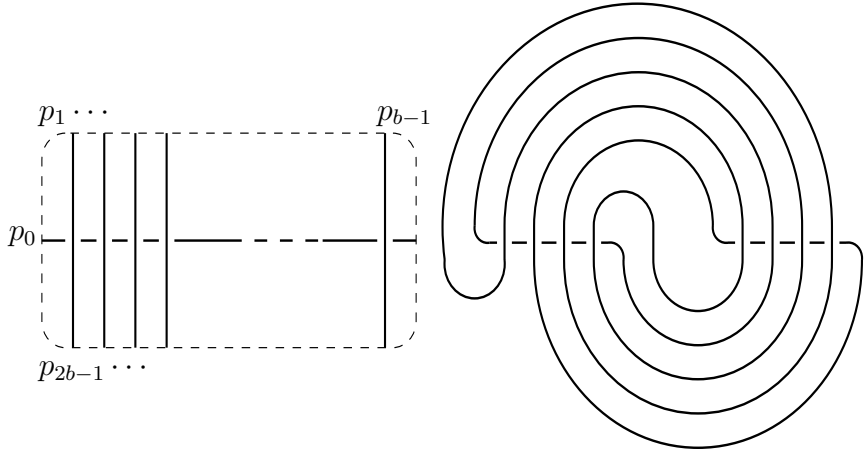
Let $H_{a,b}$ be the fundamental group of the complement of $B(a,b)$. Then it has the following presentation: $H_{a,b}=\langle u,v| wu=vw\rangle$.
In this formula, $w=u^{e_1}v^{e_2}\cdots v^{e_{a-1}}$, where for all $k$, we set $e_k=(-1)^{\lfloor kb/a\rfloor}$.

The proof is based on Wirtinger presentation of knot groups. Removing the two under-bridges, that is the two copies of the segment joining $p_0$ to $p_b$, we obtain two disjoint arcs which correspond to the generators $u$ and $v$. Using Wirtinger relation at each crossing, one can label all the remaining arcs in the projection and the labeling is consistent providing that the relation $wu=vw$ is satisfied. This explains the presentation of $H_{a,b}$. For a precise proof, see \cite{bz,sch}.

As usual, a representation $\rho$ is determined by the traces $x=\tr \rho(u)=\tr \rho(v)$ (because $u$ and $v$ are conjugate) and $y=\tr\rho(uv)$. The equality $uw=wv$ is equivalent to $\tr(uwv^{-1}w^{-1})=2$ and this equality converts into a polynomial in $x$ and $y$ thanks to the following lemma.
\begin{lemma}
For all $A,B\in \su$ one has $\tr(AB)+\tr(AB^{-1})=\tr(A)\tr(B)$.
For any word $W$ in $A,B$, there is a polynomial in three variables $F_{W}$ such that $\tr(W)=F_{W}(\tr(A),\tr(B),\tr(AB))$.
\end{lemma}
\begin{proof}
The Cayley-Hamilton identity gives $B^2-\tr(B)B+1=0$. Multiplying by $AB^{-1}$ and taking the trace, we get the first identity.
We prove the second assertion recursively on the length of $W$ by applying the first identity in a convenient way, see for instance \cite{cs}.
\end{proof}

We finally proved that there exists a family of polynomials $F_{a,b}\in \Z[x,y]$ such that
$\boM(H_{a,b})=\{x,y\in\R^2, F_{a,b}(x,y)=0, x^2-2\le y\le 2\}$.
The inequality $ x^2-2\le y\le 2$ is the trace of the inequalities we viewed in the case of $S^1\vee S^1$. 
The equality $y=2$ holds if and only if $uv=1$ and $y=x^2-2$ if and only if $uv^{-1}=1$. This last equality occurs precisely for abelian representations of $H_{a,b}$.

The next proposition simplifies the computation of $F_{a,b}$.
\begin{proposition}[T.Q.T.Le]
Let $w$ be the word associated to $H_{a,b}$ and set $w_n$ be the word $w$ with the $n$ first and $n$ last letters removed.
Then $F_{a,b}=\sum_{n=0}^{(a-1)/2} (-1)^n F_{w_n}$ where we set $F_{1}=1$.
\end{proposition}
For instance, the figure eight knot 4.1 is $B(5,3)$ and then one compute $w=uv^{-1}u^{-1}v$ and $F_{5,3}=x^2y-y^2-2x^2+3$.
The torus knot $T(5,2)=5.1=B(5,1)$ has $w=uvuv$ and $F_{5,1}=y^2-y-1$.
One recover the corresponding representation spaces on Figure \ref{2pontrep}.
\begin{figure}[h]\label{2pontrep}
\begin{center}
\includegraphics[width=12cm]{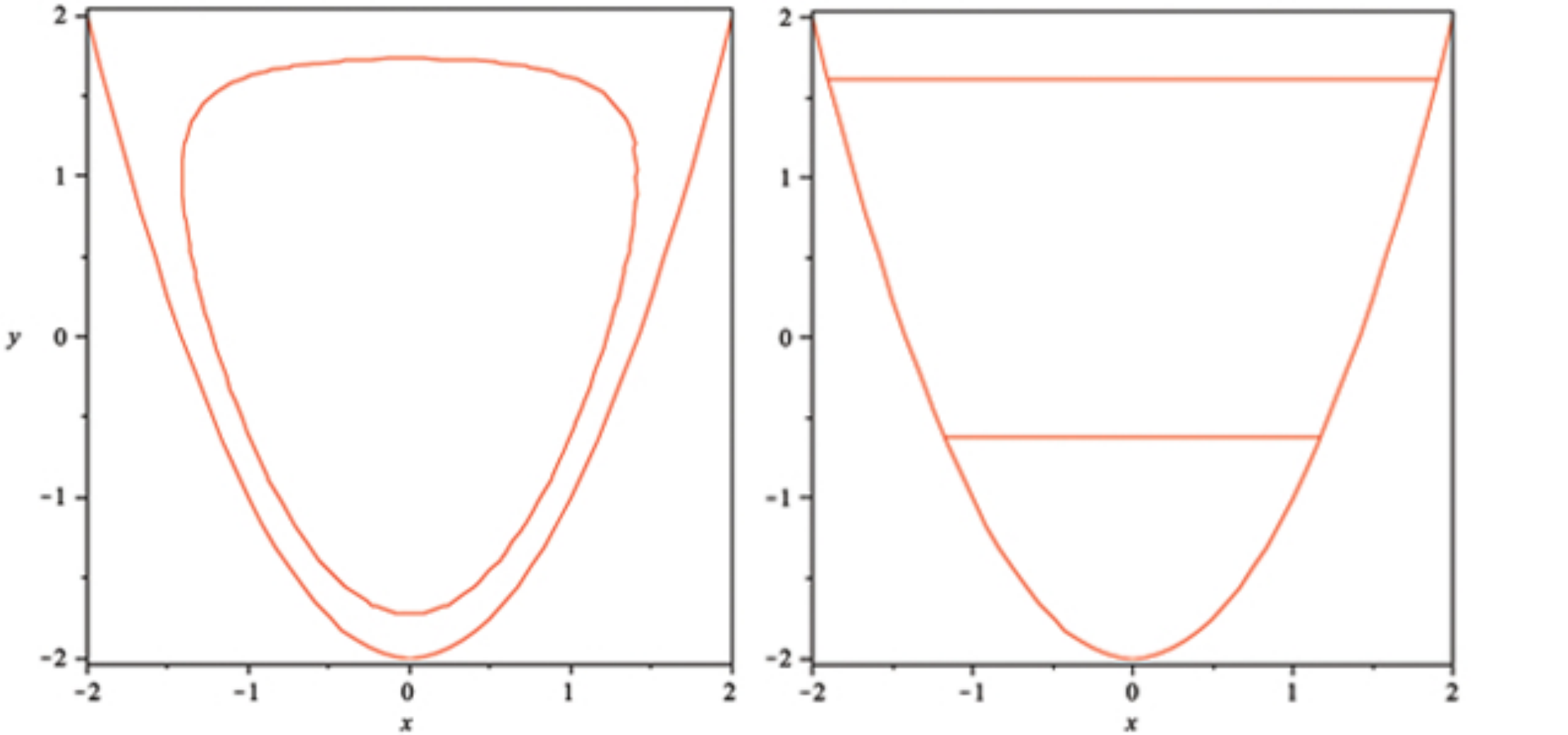}
\caption{Moduli space of 4.1 and 5.1}
\end{center}
\end{figure}

\section{Differentiable structure and twisted cohomology}
Let $X$ be a topological space whose fundamental group is finitely generated. Let $\pi_1(X)=\langle t_1,\ldots,t_n|R_1,\ldots,R_m\rangle$ be a presentation.
We recall that $\boM(X)$ is the quotient of $\boR(X)$ by an action of $\su$ and that $\boR(X)$ is identified to the preimage of $1$ by the map $R:\su^n\to\su^m$ defined by $
R(A_i)=(R_j(A_i))$. The latter space will be a submanifold of $\su$ provided that $R$ is a submersion on the preimage of $1$. The purpose of this chapter is to use this argument in a systematic way.

In what concerns the quotient, we see that the stabilizer of a representation $\rho$ has the following form:
\begin{enumerate}
\item $\su$ if $\rho$ is central.
\item $S^1$ if $\rho$ is abelian.
\item $\{\pm 1\}$ if $\rho$ is irreducible.
\end{enumerate}
For a good geometric quotient, we will need the stabilizer to be constant and moreover, the biggest part of the moduli space will correspond to the smallest stabilizer. All these conditions will be easily readable in the twisted cohomology we introduce now.

\begin{definition}
Let $W$ be a finite CW-complex with a 0-cell as base point. Denote by $\tilde{W}$ the universal covering of $W$ and by $A$ the ring $\Z[\pi_1(W)]$.
Then, the cellular complex $C_*(\tilde{W},\Z)$ is naturally a left A-module where $\pi_1(W)$ acts by deck transformations. 

Given any left $A$-module $E$, we define
\begin{equation*}\begin{split}
 C^*(W,E)=\Hom_A(C_*(\tilde{W}),E)\text{ and }H^*(W,E)=H^*\left(C^*(W,E)\right)\\
 C_*(W,E)=C_*(\tilde{W})\otimes_A E\text{ and }H_*(W,E)=H_*\left(C_*(W,E)\right)
\end{split}\end{equation*}
\end{definition}

\subsection{Two examples}

Let us look at the case $W=S^1$.  Its universal cover is $\R$ with $\Z$ acting by translations. We define $\te_0=\{0\}$ and $\te_1=[0,1]$. These cells project respectively on the 0-cell and the 1-cell of $S^1$. As $A$-modules we have
$C_0(\tilde{S^1})=A.\te_0$ and $C_1(\tilde{S^1})=A.\te_1$. Identifying $A$ with $\Z[t^{\pm 1}]$, we compute that $\partial \te_1=(t-1)\te_0$.

An A-module is nothing more than an abelian group $E$ with an automorphism $\phi$ corresponding to the action of $t$. The twisted (co)homology of the circle is then computed from the following complexes:
$$\xymatrix{
C^0(S^1,E)\simeq E \ar[r]^{\dd}& C^1(S^1,E)\simeq E \\
C_0(S^1,E)\simeq E & C_1(S^1,E)\simeq E\ar[l]_{\partial}
}$$
where $\partial v=\phi(v)-v$ and $\dd$ is obtained from the formula $\dd\lambda=(-1)^{|\lambda |+1} \lambda\circ\partial$.
We deduce from it the isomorphisms $H^0(S^1,E)=H_1(S^1,E)=\ker(\id-\phi)$ and $H^1(S^1,E)=H_0(S^1,E)=\coker(\id-\phi)$.

As a first application, consider $\boM(S^1)$, that is the set of conjugacy classes in $\su$.
Pick $g\in \su$ not central, and consider the map $c_g:\su\to \su$ defined by $c_g(h)=hgh^{-1}$.

We identify once for all the tangent space of $\su$ at $g$ to $\lu$ via the map which associates to a path $g_t$ such that $g_0=g$ the derivative $\frac{\dd}{\dd t}|_{t=0}g_tg_0^{-1}$.

Using this identification, we compute $D_1 c_g(\xi)=\frac{\dd}{\dd t}|_{t=0} e^{t\xi}ge^{-t\xi}g^{-1}=\xi-\Ad_g\xi$.
From this computation, we see that the tangent space of the $\su$-orbit through $g$ is the image of the map $\xi-\Ad_g\xi$. Hence, the tangent space of $\boM(S^1)$ at $[g]$ is the cokernel of this map. We can interpret it as $H^1(S^1,\Ad_g)$ where $\Ad_g$ is a notation for the vector space $\lu$ with automorphism $\Ad_g$.

This point is indeed very general. Before attacking the general case, let us look at the case of $W=S^1\vee\cdots\vee S^1$ a pointed union of $n$ circles. Then, its fundamental group is free, say $\pi_1(W)=\langle t_1,\ldots,t_n\rangle$ and $\boR(W)=\su^n$.

We have that $\boM(W)=\su^n/\so$ and the action is free on irreducible representations. 

Hence for an irreducible representation $\rho$, one has the following identification : $T_{[\rho]}\boM(W)=\coker D_1c_{\rho}$ where $c_{\rho}(h)=(h\rho_1h^{-1},\ldots,h\rho_n h^{-1})$ and $\rho_i=\rho(t_i)$.
We compute as before $D_1c_{\rho}(\xi)=(\xi-\Ad_{\rho_1}\xi,\ldots,\xi-\Ad_{\rho_n}\xi)$.

On the other hand, $\tilde{W}$ is a regular $2n$-valent tree, and one can chose $\te_0$ as a lift of the base points and oriented edges $\te_1^1,\ldots,\te_1^n$ starting from $\te_0$ and representing $t_1,\ldots,t_n$. As before, we have $\partial \te_1^i=(t_i-1)\te_0$. We define again the $A$-module $\Ad_{\rho}$ as $\lu$ with $t_i$ acting as $\Ad_{\rho_i}$
$$C^0(W,\Ad_{\rho})=\lu, C^1(W,\Ad_{\rho})=\lu^n\text{ and }\dd \xi= D_1 c_{\rho}(\xi).$$

At irreducible representations, we compute $H^0(W,\Ad \rho)=\ker \dd^0=\{\xi\in \lu, \Ad_{\rho_i} \xi=\xi\}=\{0\}$ as $\Ad\rho$ is irreducible as a representation of $\pi_1(W)$ in $\lu$.
Hence, as before we have an identification $T_{[\rho]}\boM(W)=H^1(W,\Ad_{\rho})$.

\subsection{The general case}
Suppose $W$ is a 2-dimensional CW-complex of the following form: 
\begin{enumerate}
\item $1$ 0-cell lifted to $\te_0$
\item $n$ 1-cells lifted to oriented edges $\te_1^i$ for $i\le n$ starting at $\te_0$.
\item $m$ 2-cells lifted to polygons $\te_2^j$ for $j\le m$ with a base point at $\te_0$.
\end{enumerate}
For each 2-cell $\te_2^j$, one can read starting at $\te_0$ a word in generators $t_i$ represented by the 1-cells. Denoting by $R_j$ these words, we get a presentation of $\pi_1(W)$ given by
$$\pi_1(W)=\langle t_1,\ldots,t_n| R_1,\ldots R_m\rangle.$$
Let $R:\su^n\to\su^m$ be the map defined for a n-tuple $\rho=(\rho_1,\ldots,\rho_n)$ by $R(\rho_1,\ldots,\rho_n)=(R_1(\rho),\ldots,R_m(\rho))$. The space $\boR(W)=R^{-1}(1,\ldots,1)$ is smooth at $\rho$ if $R$ is a submersion at $\rho$. Let us compute the differential of $R$ by supposing $m=1$. We write $R(\rho)=\rho_{i_1}^{\epsilon_1}\cdots \rho_{i_k}^{\epsilon_k}$ for $i_l\in\{1,\ldots,n\}$ and $\epsilon_l=\pm 1$.

\begin{eqnarray*}\label{derivee}
D_{\rho} R(\xi_i)&=&\frac{\dd}{\dd t}|_{t=0} (e^{t\xi_1}\rho_{i_1})^{\epsilon_1}\cdots (e^{t\xi_k}\rho_{i_k})^{\epsilon_k}\\
&=&\sum_l \left\{
\begin{array}{c}
\rho_{i_1}^{\epsilon_1}\cdots \rho_{i_{l-1}}^{\epsilon_{l-1}}\xi_{i_l}\rho_{i_l}^{\epsilon_l}\cdots \rho_{i_k}^{\epsilon_k}\quad\text{ if }\epsilon_{i_l}=1\\
-\rho_{i_1}^{\epsilon_1}\cdots \rho_{i_{l}}^{\epsilon_{l}}\xi_{i_l}\rho_{i_{l+1}}^{\epsilon_{l+1}}\cdots \rho_{i_k}^{\epsilon_k}\quad\text{ if }\epsilon_{i_l}=-1
\end{array}
\right.
\end{eqnarray*}
On the other side, $\partial \te_2$ is the sum over $l$ of either $t_{i_1}^{\epsilon_1}\cdots t_{i_{l-1}}^{\epsilon_{l-1}} \te_1^{i_l}$ if $\epsilon_l=1$ or $-t_{i_1}^{\epsilon_1}\cdots t_{i_{l}}^{\epsilon_{l}} \te_1^{i_l}$ if $\epsilon_l=-1$.
Taking the adjoint of this map to obtain a map from $C^1(W,\Ad \rho)$ to $C^2(W,\Ad \rho)$, we get exactly the same expression as in \ref{derivee}.
In conclusion, the following diagram commutes:
$$\xymatrix{
T_{\rho} R(W)\ar[r]^{D_{\rho}R}\ar[d]^{\sim}& \lu^m\ar[d]^{\sim}\\
C^1(W,\Ad_{\rho})\ar[r]^{\dd^1}& C^2(W,\Ad_{\rho})
}$$
The map $R$ is a submersion at $\rho$ if and only if $\dd^1$ is surjective, or equivalently if $H^2(W,\Ad_{\rho})=0$. The argument on the pointed union of circles repeats exactly and shows that the action of $\su$ by conjugation at $\rho$ is locally free if $\dd^0$ is injective which amounts to say that $H^0(W,\Ad_{\rho})=0$. If both conditions are satisfied, then the quotient $\boM(W)$ is a manifold at $[\rho]$ and the tangent space identifies to $\ker \dd^1/\im \dd^0=H^1(W,\Ad_{\rho})$.

\subsection{Applications}
The interest of the language of (co)-homology is to use its tools, namely exact sequences, Poincar\'e duality and universal coefficients.

First, we can define relative (co)-homology of pairs as in the untwisted case and it fits into a long exact sequence as usual. In what concerns Poincar\'e duality, we have the following generalization: given a compact and oriented $n$-manifold with boundary and a $A$-module $E$, the cap product with the fundamental class $[M]$ gives an isomorphism
$$H^k(M,E)\simeq H_{n-k}(M,\partial M; E)$$
In what concerns universal coefficients, let $W$ be a finite CW-complex, $R$ be a principal ring and $E$ a $R[\pi_1(W)]$-module which is free as a $R$-module. Then there is an exact sequence:

$$0\to \text{Ext}(H_{k-1}(W,E),R)\to H^k(W,E^*)\to H_k(W,E)^*\to 0.$$
The proof is the usual one applied to the free complex $C_*(W,E)$ whose dual identifies to $C^*(W,E^*)$ thanks to our hypotheses.
We will use very often the well-known fact that the Euler characteristic of a complex and of its homology are the same. We deduce from it that the Euler characteristic of $H^*(W,\Ad\rho)$ is $3\chi(W)$ because the twisted complex is obtained from the standard one by tensoring by $\lu$ which has dimension 3.
\subsubsection{Surfaces}
Let $\Sigma$ be a closed surface and $\rho\in \boR(\Sigma)$ be an irreducible representation. The $\R[\pi_1(\Sigma)]$-module $\Ad_{\rho}$ is free over $\R$ of dimension 3 and its dual identifies to itself thanks to the invariant Killing form $\langle A,B\rangle=\tr A\ba{B}^T$.
We deduce that $H^2(\Sigma,\Ad_{\rho})\simeq H_0(\Sigma,\Ad_{\rho})\simeq H^0(\Sigma,\Ad_{\rho}^*)^*\simeq H^0(\Sigma,\Ad_{\rho})^*=0$. 
Hence irreducible representations are smooth points of $\boM(\Sigma,\su)$.
Written differently, Poincar\'e duality states that the the following pairing is non degenerate.

$$\xymatrix{
H^1(\Sigma,\Ad_{\rho})\times H^1(\Sigma,\Ad_{\rho})\ar[r]^-{\cup}&
H^2(\Sigma,\Ad_{\rho}\otimes\Ad_{\rho})\ar[r]^-{\langle,\rangle}&
H^2(\Sigma,\R)\ar[r]^-{\int}&
\R}$$
This gives a non-degenerate 2-form $\omega$ on the irreducible part of $\boM(\Sigma)$. We will show later that it is a closed form, and hence that $\boM(\Sigma)$ is symplectic.

\subsubsection{The torus case}\label{torusdiff}
This case is not covered by the previous one because all representations $\rho$ on a torus are abelian. Before computing the corresponding cohomology group, recall that $\boM(S^1\times S^1)$ is covered by the map sending a pair $(\phi,\psi)$ to the representation $\rho_{\phi,\psi}$ sending the first generator to $\rho_{\phi}=\begin{pmatrix}e^{i\phi}&0\\ 0&e^{-i\phi}\end{pmatrix}=e^{{\bf i}\phi}$ and the second one to $\rho_{\psi}=\begin{pmatrix}e^{i\psi}&0\\ 0&e^{-i\psi}\end{pmatrix}=e^{{\bf i}\psi}$.

Then $H^0(S^1\times S^1,\Ad_{\rho})\simeq H^2(S^1\times S^1,\Ad_{\rho})^*\simeq \ker (\id -\rho_{\phi})\cap\ker(\id-\rho_{\psi})$.
In the quaternionic basis $\bf{i},\bf{j},\bf{k}$, $\Ad_{\rho_{\phi}}=\begin{pmatrix} 1&0&0\\
0& \cos(2\phi)& -\sin(2\phi)\\ 0& \sin(2\phi)& \cos(2\phi)\end{pmatrix}$.
The subspace fixed by this matrix is generated by $\bf{i}$ provided that $2\phi\notin 2\pi\Z$.

We deduce from it that the rank of $H^2(S^1\times S^1,\Ad_{\rho})$ is constant equal to 1 if $\phi$ or $\psi$ is not in $\pi\Z$, or equivalently if $\rho_{\phi,\psi}$ is not central.
One can apply the constant rank theorem to state that $\boM(S^1\times S^1)$ is actually a manifold at all non central representations. A computation shows that the pull back of $\omega$ in the coordinates $\phi,\psi$ is $\dd \phi\wedge \dd\psi$.

\subsubsection{3-manifolds with boundary}

Let $M$ be a 3-manifold with boundary, and $\rho\in \boR(M)$.  The following exact diagram is a piece of the sequence of the pair $(M,\partial M)$, where vertical isomorphisms are given by Poincar\'e duality.

$$\xymatrix{
H^1(M,\Ad_{\rho})\ar[r]^-{\alpha}\ar[d]^{\sim}&H^1(\partial M,\Ad_{\rho})\ar[r]^-{\beta}\ar[d]^{\sim}&H^2(M,\partial M;\Ad_{\rho})\ar[d]^{\sim}\\
H^2(M,\partial M;\Ad_{\rho})^*\ar[r]^-{\beta^*}&H^1(\partial M,\Ad_{\rho})^*\ar[r]^-{\alpha^*}&H^1(M,\Ad_{\rho})^*
}$$
We read from this diagram that $\rk \beta=\rk\alpha^*$. Standard linear algebra says that $\rk\alpha^*=\rk\alpha$ and the exactness of the first line gives $\rk\alpha=\dim\ker\beta$.

One deduce from it that $\dim H^1(\partial M,\Ad_{\rho})=\rk\beta+\dim\ker\beta=2\rk\alpha$ whereas $\rk H^1(M,\Ad_{\rho})=\rk \alpha+\dim\ker\alpha=\frac{1}{2}\dim H^1(\partial M,\Ad_{\rho})+\dim\ker \alpha$.

It implies that the three following properties are equivalent:
\begin{enumerate}
\item The map $H^1(M,\Ad_{\rho})\to H^1(\partial M,\Ad_{\rho})$ is injective.
\item The map $H^1(M,\partial M;\Ad_{\rho})\to H^1(M,\Ad_{\rho})$ vanishes.
\item $\dim H^1(M,\Ad_{\rho})=\frac{1}{2}H^1(\partial M;\Ad_{\rho})$.
\end{enumerate}
We deduce from these computations the following result.

Let $\rho\in \boR(M)$ a representation. We will call it {\it regular} if it is irreducible and satisfies the equivalent properties above.
\begin{theorem}
 Regular representations are smooth points of $\boR(M)$ and 
the restriction map $\boM(M)\to \boM(\partial M)$ is a Lagrangian immersion when restricted to regular representations and corestricted to irreducible (non central in the torus case) representations.
\end{theorem}
\begin{proof}

Let $\rho$ be a regular representation in $\boR(M)$, and consider a CW-complex $W$ on which $M$ retracts. Then, the differential $\dd^1$ in the complex $C^*(W,\Ad_{\rho})$ is the derivative at $\rho$ of equations defining $\boR(M)$. The corank of this derivative is equal to the dimension of $H^2(W,\Ad_{\rho})=H^2(M,\Ad_{\rho})$ : we will see that the conditions for a representation to be regular are equivalent to the condition that $H^2(M,\Ad_{\rho})$ is as small as possible, and hence will ensure that $\boR(M)$ is smooth at $\rho$.

As $\rho$ is irreducible, $H^0(M,\Ad_{\rho})=0$ and the computation of Euler characteristic gives $\chi(M)\dim \lu=-\dim H^1(M,\Ad_{\rho})+\dim H^2(M,\Ad_{\rho})$. We deduce from that formula that $H^2(M,\Ad_{\rho})$ is as small as possible if and only if the same is true for $H^1(M,\Ad_{\rho})$.
We have proved that $\dim H^1(M,\Ad_{\rho})=\frac{1}{2}\dim H^1(\partial M)+\dim \ker \alpha$ where $\alpha$ is the natural map from $H^1(M,\Ad_{\rho})$ to $H^1(\partial M,\Ad_{\rho})$.

By assumption, the restriction of $\rho$ to the boundary is in the smooth part of $\boR(M)$, hence the dimension of $H^1(\partial M,\Ad_{\rho})$ do not change. The condition of regularity is $\dim \ker \alpha=0$ and hence is equivalent to the fact that the dimension of $H^1(M,\Ad_{\rho})$ is as small as possible. We proved the first part of our assumption.
We also see that the equation $\dim\ker \alpha=0$ implies that $\dim H^1(M,\Ad_{\rho})=\frac{1}{2}\dim H^1(\partial M, \Ad_{\rho})$. Moreover, $\alpha$ is the derivative at $[\rho]$ of the restriction map $r:\boM(M)\to \boM(\partial M)$ which becomes an immersion.
Finally, let us show that the image of $\alpha$ is isotropic by looking at the diagram above.
Let $u,v\in H^1(M,\Ad_{\rho})$. Then, one may show that $\omega(\alpha(u),\alpha(v))=0$. But, this can be written as $\langle \beta_*(u^{\#}), \alpha(v)\rangle$ where $u^{\#}$ is Poincar\'e dual to $u$ and $\langle\cdot,\cdot\rangle$ is the duality pairing. But recall that $\im \beta_*=(\ker \beta)^{\perp}=(\im\alpha)^{\perp}$ this implies precisely that $\omega(\alpha(u),\alpha(\beta))$ vanishes. Hence, the image of $\alpha$ is Lagrangian and the theorem is proved.
\end{proof}

Let us look at the following example. As seen in Section \ref{noeuds}, the moduli space of the trefoil knot is a segment (abelian representations) with another segment attached to it. The restriction map sends this space to the representation space of the torus that we represent as a union of two stacked up squares (known as the pillow case). The image of this map is represented on the left of Figure \ref{embed}. The restriction map is an immersion for all regular points but is not injective. The case of the figure eight knot is presented on the right. In that case, the moduli space is the union of a segment (abelian representations) and a circle (irreducible representations). The restriction map is an immersion at all regular points and the map fails to be injective at one point.
\begin{figure}[h]\label{embed}
\begin{center}
\includegraphics[width=8cm]{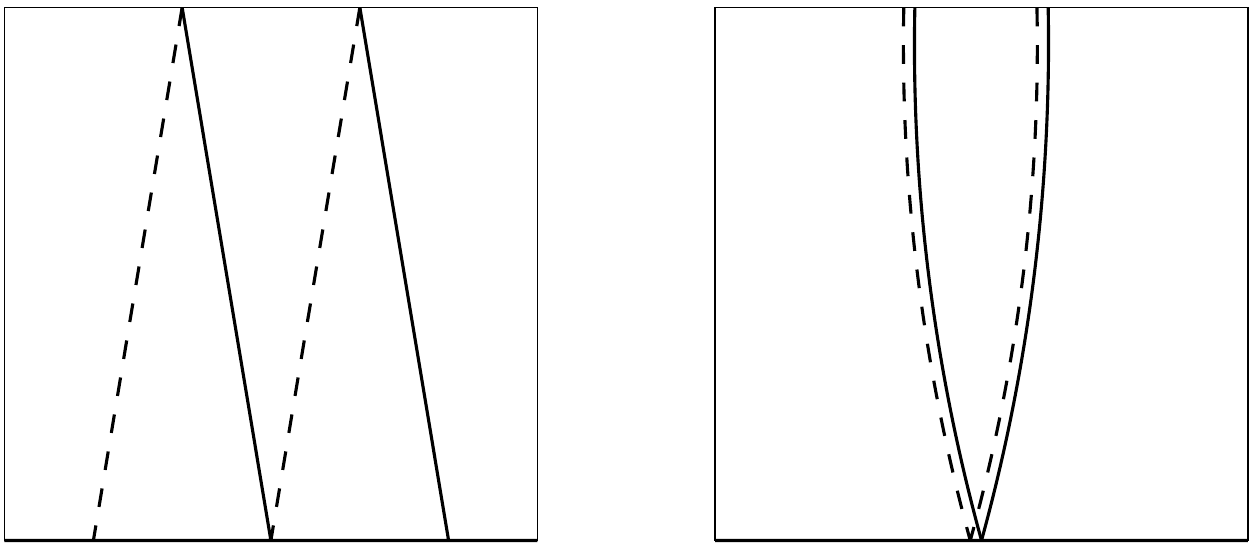}
\caption{Moduli spaces of 3.1 and 4.1 restricted to the boundary}
\end{center}
\end{figure}

\subsection{Reidemeister torsion}
Given a real vector space $V$ of dimension $n$, we can form the line $\det V=\Lambda^n V$. We will consider this line as even if $n$ is even and odd if $n$ is odd. This convention will play a role with the implicit use of the isomorphism $V\otimes W\to W\otimes V$ sending $v\otimes w\to  (-1)^{|v||w|}w\otimes v$ where $|v|$ and $|w|$ are the degrees of $v$ and $w$ respectively. With that convention, one can write $\det(V\oplus W)=\det(V)\otimes \det(W)$ as this identification depends on the order of $V$ and $W$ up to a sign prescribed by the convention.

The following considerations will rely on the fact that for any short exact sequence
$0\to U\to V\to W\to 0$ there is a canonical isomorphism $\det V=\det U\otimes \det W$. 
This isomorphism is defined by sending $u_1\wedge\cdots \wedge u_i\otimes w_1\wedge\cdots \wedge w_j$ to $u_1\wedge\cdots \wedge u_i\wedge\tilde{w}_1\wedge\cdots \wedge\tilde{w}_j$ where $\tilde{w}$ is any lift of $w$ in $V$.

Given a finite complex of finite dimensional real vector spaces $C^*$, we define $\det C^*=\det C^0\otimes (\det C^1)^{-1}\otimes \cdots \otimes (\det C_k)^{(-1)^k}$. We define the determinant of its cohomology by the same formula.
\begin{lemma}\label{lem-torsion}
There is a canonical isomorphism (which involves the sign rule)
$$\det C^*=\det H^*.$$
\end{lemma}
\begin{proof}
Define for all $i\in \Z$, $Z_i=\ker \dd^i$ and $B^i=\im \dd^i$. The exact sequence
$0\to Z^i\to C^i\to B^i\to 0$ gives $\det C^i=\det Z^i\otimes \det B^i$ and the exact sequence
$0\to B^{i-1}\to Z^i\to H^i\to 0$ gives $\det Z^i=\det B^{i-1}\otimes \det H^i$.
Removing the $\otimes$ sign, we compute:
\begin{eqnarray*}
\det C^0(\det C^1)^{-1}\det C^2\cdots=\det Z^0 \det B^0 (\det Z^1\det B^1)^{-1}\det Z^2\det B^2\cdots\\
=\det B^{-1} \det H^0 \det B^0(\det B^0\det H^1\det B^1)^{-1}\det B^1\det H^2\det B^2\cdots
\end{eqnarray*}
One see that all factors $\det B^{i}$ appear twice with opposite signs and then cancel, proving the proposition.
\end{proof}

Let $W$ be a finite CW-complex and $\rho$ be a representation in $\boR(W)$. For all cells, we choose an orientation and a lift to the universal covering $\tilde{W}$. For each $k$, number the lifted $k$-cells as $\te_k^{1},\ldots,\te_k^{n_k}$.

Recall that the twisted cochain complex of dimension $k$ is $C^k(W,\Ad_{\rho})=\bigoplus_{i\le n_k} \te_k^i \lu$. Choose in $\Lambda^3\lu$ the $\su$-invariant volume element $\lambda={\bf i}\wedge {\bf j}\wedge {\bf k}$. We define again $\lambda_k=\bigwedge_{i\le n_k} \te_k^i\lambda$, and $\Lambda=\lambda_0\otimes \lambda_1^{-1}\otimes \cdots\in \det C^*(W,\Ad_{\rho})$.

With the isomorphism provided by the lemma, it gives us an element in $\det H^*(W,\Ad_{\rho})$ that we denote by the same letter. One can check easily that this element does not depend on the choice of lifting of the cell because $\su$ fixes $\lambda$, but changing the orientation of a cell or their numbering do change $\Lambda$ up to a sign.
To remove this ambiguity, we do the same construction for the complex $C^*(W,\R)$ by taking the same cells with the same orientation and order, and replacing $\lambda$ with $1\in \det \R$.
Doing so, one gets an element $T\in \det C^*(W,\R)=\det H^*(W,\R)$. We call Reidemeister torsion at $\rho$ the quotient $W/T\in \det H^*(W,\Ad_{\rho}) \det H^*(W,\R)^{-1}$.

It remains to understand how this quotient changes when we change the cellular decomposition but one can show that it does not change under cellular subdivision and collapsing, see \cite{turaev}. This implies that the torsion only depends on $W$ up to simple homotopy.

Let us give some applications: in the case of a surface $\Sigma$ and irreducible $\rho\in \boR(\Sigma)$, one has $H^0(\Sigma,\R)=\R$, $H^2(\Sigma,\R)=\R$, $H^0(\Sigma,\Ad_{\rho})=H^2(\Sigma,\Ad_{\rho})=0$. By considering the standard generators of the determinant of these spaces, one see that the torsion reduces to an element of $\det H^1(W,\Ad_{\rho})(\det H^1(W,\R))^{-1}$. This element does not give us any information as one can show that it is equal to the Liouville volume form $\left(\frac{\omega_{\rho}^{3g-3}}{(3g-3)!}\right)^{-1}\left(\frac{\omega^g}{g!}\right)$ where $g$ is the genus of $\Sigma$ (supposed at least equal to 2) and $\omega, \omega_{\rho}$ are the symplectic forms on $H^1(\Sigma,\R)$ and $H^1(\Sigma,\Ad_{\rho})$ respectively.

 In the case of a regular representation in $\boR(M)$, we have 
 $H^0(M,\Ad_{\rho})=0$ and $H^2(M,\Ad_{\rho})=H^2(\partial M,\Ad_{\rho})$ as one can deduce from the commutative diagram below with exact lines:

$$\xymatrix{
H^2(M,\partial M)\ar[r]\ar[d]^{\sim}&H^2(M)\ar[r]\ar[d]^{\sim}&H^2(\partial M)\ar[r]&H^3(M,\partial M)=0\\
H^1(M)^*\ar[r]^-{0}&H^1(M,\partial M)^*&&
}$$
 Then, one has $H^2(\partial M)=H^0(\partial M)^*$. This is 0 if the genus of $\partial M$ is greater than 2, and in the case of a torus, it is one dimensional. One can choose a preferred generator of this space, showing that $\det H^*(M,\Ad_{\rho})\simeq H^1(M,\Ad_{\rho})^{-1}$.
 Choosing an element $T$ in $\det H^*(M,\R)$, one gets a well-defined element $\Lambda\in H^1(M,\Ad_{\rho})^{-1}$ which may be interpreted as a volume form on the regular part of $\boM(M)$.
 
\section{Gauge theory}

\subsection{Principal bundles and flat connections}
Let $M$ be a compact manifold of dimension at most 3. We will denote $\su$ by $G$ as it can be replaced by any Lie group in that section. The Lie algebra of $G$ will be denoted by $\g$.
\begin{definition}
A principal $G$-bundle over $M$ is a fiber bundle $\pi:P\to M$ with a right action of $G$ on $P$ such that $G$ acts freely and transitively on each fiber. Two such bundles are isomorphic if there is a $G$-equivariant bundle isomorphism lifting the identity of $M$.
\end{definition}

\begin{definition}
A principal bundle with flat structure $(P,\boF)$ is a principal bundle $\pi:P\to M$ and a foliation $\boF$ of $P$ which is $G$-equivariant and such that the restriction of $\pi$ to each leaf is a local diffeomorphism. Again two such pairs are isomorphic if the bundles are isomorphic and the foliations correspond through this isomorphism.
\end{definition}
Such a flat $G$-bundle is often described by covering $M$ with open sets $U_i$ on which we can find sections $s_i$ of $P$ whose image lie in the same leaf (we will say that these sections are flat). On the intersection of two such open sets $U_i$ and $U_j$, the two section differ by the action of a locally constant map $g_{i,j}:U_i\cap U_j\to G$. These data are sufficient to reconstruct the flat $G$-bundle up to isomorphism. 

Take a point $p$ in a flat $G$-bundle $P$ and write $x=\pi(p)$. Then, any path $\gamma:[0,1]\to M$ such that $\gamma(0)=p$ lifts uniquely to a path $\tilde{\gamma}\to P$ if one asks that $\tilde{\gamma}(0)=p$ and that $\tilde{\gamma}$ stays locally on the same leaf. If $\gamma(1)=x$, then $\tilde{\gamma}(1)=pg$. The assignement $\gamma\to g$ do not depend on the homotopy class of $\gamma$ and gives rise to an homomorphism $\hol_p:\pi_1(M,x)\to G$.

The conjugacy class of this representation do not depend on $p$ and $x$ moreover, we have the following fundamental result:

\begin{theorem}
The holonomy map gives a bijection between isomorphism classes of flat $G$-bundles and $\boM(M,G)$.
\end{theorem}
\begin{proof}
We construct the reverse map  in the following way. Let $x$ be a base point in $M$ and $\rho$ a representation of $\pi_1(M,x)$. Then, denote by $M\times_{\rho}G$ the quotient of $\tilde{M}\times G$ by the equivalence relation $(m,g)\sim (\gamma.m,\rho(\gamma)g)$ for all $m\in M,g\in G, \gamma \in \pi_1(M,x)$. The map $\pi(m,g)=m$ and the action $(m,g).h= (m,gh)$ give to $M\times_{\rho}G$ a $G$-bundle structure. The foliation $\boF$ is the quotient of the foliation of
$\tilde{M}\times G$ whose leaves are $\tilde{M}\times\{g\}$ for $g\in G$.
One can check that these constructions are reciprocal, which proves the theorem.
\end{proof}

\subsection{Sections and connection forms}
On manifolds $M$ of dimension at most 3 and for connected and simply connected groups $G$, all $G$-bundles on $M$ are trivial, that is isomorphic to $M\times G$. To prove this, it is sufficient to find a section $s$ of any $G$-bundle $\pi:P\to M$. The map $M\times G\to P$ sending $(m,g)$ to $s(m)g$ will be the desired isomorphism.

Let $W$ be a CW-complex homotopic to $M$ and let $\pi:P\to W$ be a $G$-bundle. Then, one can choose arbitrarily a section over the 0-skeleton of $W$. For each 1-cell,we extend the section of $P$ given at the ends, using the fact that the fiber (isomorphic to $G$) is connected. At the boundary of each 2-cell, there is some section chosen that we can extend along the cell as $G$ is simply connected. Finally, we deduce from the fact that $\pi_2(G)=0$ that the section also extends to the 3-cells and hence to $W$, which proves the assumption.

The existence of sections give us another practical viewpoint on flat $G$-bundles that we explain now. 

Let $(P,\boF)$ be a flat $G$-bundle. Given a section $s$ of $P$, one can encode the foliation $\boF$ with a 1-form $A$ on $M$ with values in $\g$. We define it in the following way: let $x$ be a point in $M$. Let $h$ be the map defined at a neighborhood of $x$ with values in $G$ such that the map $m\mapsto s(m)h(m)$ describes the leaf of $\boF$ passing at $s(x)$ as suggested in Figure \ref{principal}. Then, we set $A_x=-D_x h\in \g$.

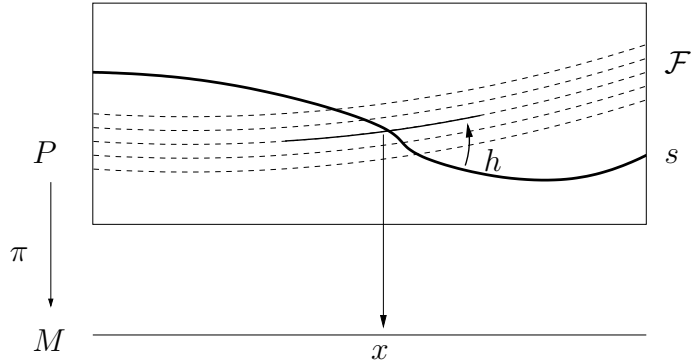
\begin{figure}[h]\label{principal}
\begin{center}
\begin{pspicture}(-2,0)(6,5)
\includegraphics[width=8cm]{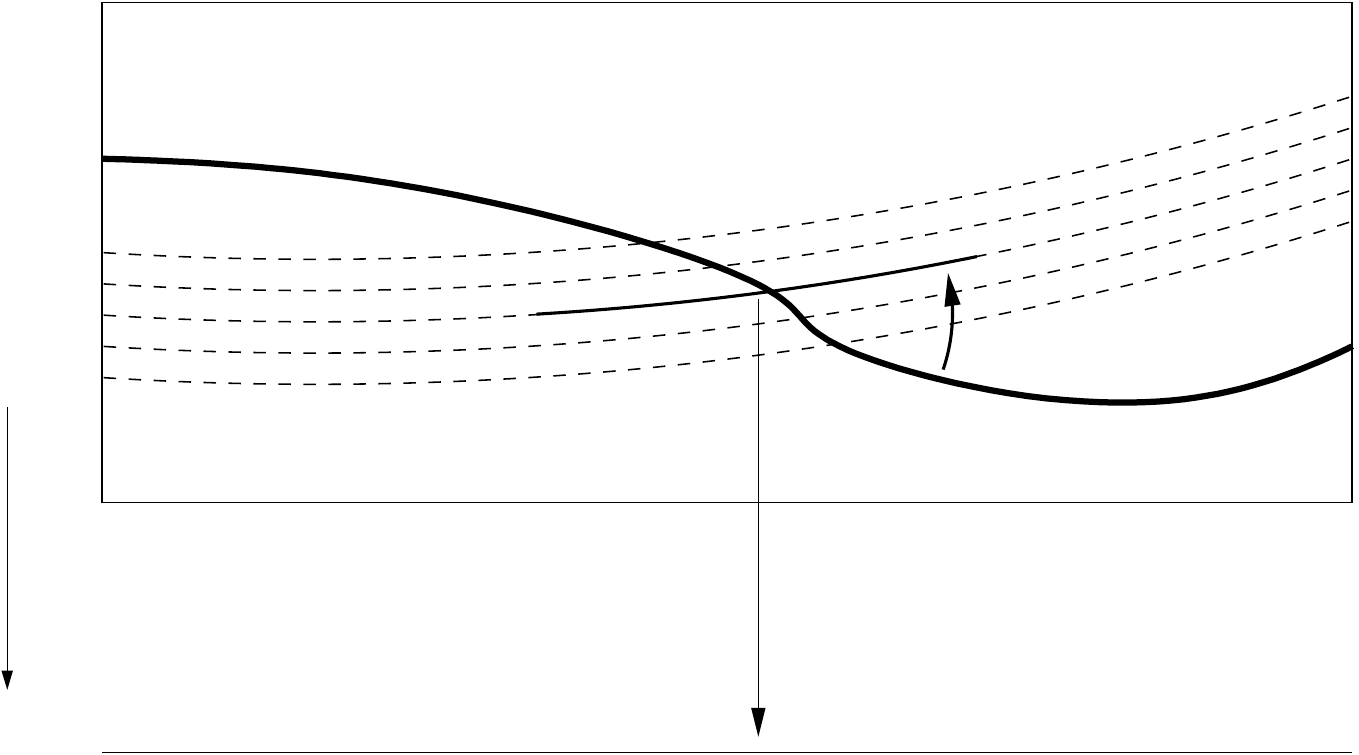}
\put(-8.5,1){$\pi$}
\put(-8.2,2.3){$P$}
\put(-8.2,-0.2){$M$}
\put(0.2,2.3){$s$}
\put(0.2,3.5){$\boF$}
\put(-3.7,-0.3){$x$}
\put(-2.2,2.2){$h$}
\end{pspicture}
\caption{From foliations to connection forms}
\end{center}
\end{figure}

Any other section $s'$ is obtained from $s$ by the right action of a map $g:M\to G$. We write for short $s'=s^g$.
If the leaf passing at $s(x)$ is the image of the map $sh$ for some map $h$ defined around $x$ with values in $G$, then by $G$-invariance, the image of the leaf passing at $s(x)g(x)$ is the image of the map $m\mapsto s(m)h(m)g(x)$. This map can be rewritten as $sgg^{-1}hg(x)$. The 1-form $A^g$ associated to the section $s^g$ is then the derivative at $x$ of $g^{-1}hg(x)$, that is $g^{-1}D_xhg-g^{-1}D_xg$. We then obtain $A^g=g^{-1}Ag+g^{-1}\dd g$.

Nevertheless, any section $s$ and 1-form $A\in\Omega^1(M,\g)$ do not necessarily define a flat structure but only a $G$-invariant distribution of subspaces in $TP$ transverse to the fibers of $\pi$. We would like to give a condition on $A$ for this to hold. In order to be tangent to a foliation, the distribution has to be integrable, that is to verify the Frobenius condition that if two vector fields belong to this distribution, their bracket also belongs to it. To verify this condition, consider two vector fields $X,Y$ on $M$.
The 1-form $A$ defines uniquely a 1-form $\tilde{A}\in \Omega^1(P,\g)$ whose kernel is the invariant distribution. This $1$-form is characterized by the equations $R_g^*\tilde{A}=g^{-1}\tilde{A}g$ and $s^*\tilde{A}=A$, where $R_g$ is the action of a fixed element $g$ on $P$.
The vector fields $X$ and $Y$ extend to unique "horizontal" vector fields $\tilde{X}$ and $\tilde{Y}$ on $P$ such that $\tilde{A}(\tilde{X})=\tilde{A}(\tilde{Y})=0$ and $\pi_{*} \tilde{X}=X,\pi_*\tilde{Y}=Y$. These vector fields belong to the distribution defined by $A$ and their bracket will belong to the distribution if and only if $\tilde{A}([\tilde{X},\tilde{Y}])=0$.
The equation $\dd\tilde{A}(\tilde{X},\tilde{Y})=\tilde{X}.\tilde{A}(\tilde{Y})-\tilde{Y}.\tilde{A}(\tilde{X})-\tilde{A}([\tilde{X},\tilde{Y}])$ give $\dd \tilde{A}(\tilde{X},\tilde{Y})=0$.

Let $\xi,\eta$ be two elements of $\g$ and $v_\xi, v_\eta$ be the vector fields coming from the infinitesimal action of $G$ on $P$, for instance $v_{\xi}(p)=\frac{\dd }{\dd t} pe^{t\xi}$. Then, by construction $\tilde{A}(v_{\xi})=\xi$ and hence $\dd \tilde{A} (v_{\xi},v_{\eta})=-[\tilde{A}(v_{\xi}),\tilde{A}(v_{\eta})]$. We deduce that the identity $\dd\tilde{A}+\frac{1}{2}[\tilde{A}\wedge\tilde{A}]=0$ is true when applied to pairs of horizontal (resp. vertical) vector fields. One check that this is again true for horizontal and vertical vector fields and hence, the identity $\dd\tilde{A}+\frac{1}{2}[\tilde{A}\wedge\tilde{A}]=0$ holds. Pulling it back by $s$, we get $\dd A+\frac{1}{2}[A\wedge A]=0$ which is the flatness equation for $A$. This is a necessary and sufficient condition for the distribution defined by $A$ to be integrable.

This considerations may be summarized in the following proposition:

\begin{proposition}
Let $M$ be a manifold and $G$ a Lie group such that all $G$-bundles on $M$ are trivial.
Then the set of isomorphism classes of flat $G$-bundles is isomorphic to the set of connections $A\in\Omega^1(M,\g)$ satisfying $\dd A+\frac{1}{2}[A\wedge A]=0$ up to the action of the gauge group given by $A^g=g^{-1}Ag+g^{-1}\dd g$.
\end{proposition}

\subsection{De Rham cohomology and isomorphisms}

Given a trivialized flat $G$-bundle $(P,s,A)$ where $s$ is a section of $P$ and $A$ is a flat connection, we define the twisted De Rham complex as the complex
$\Omega^*(M,\g)$ with the differential $\dd_A (\alpha)=\dd\alpha+[A\wedge\alpha]$. The flatness equation implies that $d_A^2=0$. We will denote by $H^*_A(M,\g)$ the cohomology of this complex. It is related to the twisted cohomology via the following De Rham theorem.

Let $M$ be a compact manifold homeomorphic to a cell complex $W$ and $(P,s,A)$ be a flat $G$-bundle. Fixing $p\in P$ over a base point $x$ in $M$ gives a holonomy representation $\rho\in \boR(\pi_1(M,x),G)$.

As the universal cover $\tilde M$ is contractible, the flat $G$-bundle induced on $\tilde M$ is trivial. Hence, there is a map $g:\tilde{M}\to G$ such that $\tilde{A}^g=0$ where $\tilde{A}$ is the connection form induced on $\tilde{M}$. Moreover, we can suppose that $g(x)=1$.
Let $\alpha\in \Omega^k(M,\g)$ a cocycle. Then the cocyle $I(\alpha)\in C^k(W,\Ad_{\rho})$ associates to a lifted k-cell $\te_k$ the integral $\int_{\te_k} g^{-1}\alpha g$.

\begin{theorem}\label{derham}
The map $I$ is a chain map which induces for all $k$ an isomorphism from $H^k_A(M,\g)$ to $H^k(M,\Ad_{\rho})$.
\end{theorem}

When $k=1$, we can interpret $I$ as the derivative of the holonomy function.  
More precisely, let $M$ be a manifold and $(P,\boF_t)$ be a 1-parameter family of foliations on the same $G$-bundle $P$. Assuming the existence of a section $s:M\to P$, the family of foliations gives a family of connection 1-forms $A_t$ satisfying $\dd A_t+\frac{1}{2}[A_t\wedge A_t]=0$.
Let $A=A_0$ and suppose that this family is smooth. Then $\alpha=\frac{\dd A_t}{\dd t}_{t=0}$ satisfies $\dd \alpha+[A\wedge\alpha]=0$. It represents an element $[\alpha]$ of $H^1_A(M,\g)$. 

Let $\gamma$ be  loop in $\pi_1(M,x)$ then $\hol_{\gamma} A_t \in G$ and $\frac{\dd}{\dd t}|_{t=0}\hol_{\gamma}A_t (\hol_{\gamma}A)^{-1}=I(\alpha)(\gamma)$.
To see this, it is sufficient to do the computation in the universal cover $\tilde{M}$ after having trivialized $A$.

Let $\Sigma$ be an oriented surface and $\rho$ be an irreducible representation corresponding to a flat connection $A$. Then, the tangent space of $\boM(\Sigma)$ at $[\rho]$ is isomorphic to $H^1(\Sigma,\Ad_{\rho})\simeq H^1_A(\Sigma,\g)$. The cup-product in cohomology corresponds to the exterior product of forms in De Rham cohomology. Then, the form $\omega_A:H^1_A(\Sigma,\g)^2\to \R$ is given by $\omega_A(\alpha,\beta)=\int_{\Sigma}\langle \alpha\wedge\beta\rangle$.
Remarking that $A$ does not appear in this formula is the key point for showing that $\omega$ is a closed 2-form. The following proposition is a main technical ingredient for relating gauge theoretical arguments to the study of representation spaces.

\begin{proposition}\label{releve}
Let $M$ be a compact manifold of dimension less than 3, and $U$ be a contractible open set in $\boM(M,\su)$ consisting of regular representations. Then, there is a smooth map $A:U\to \Omega^1(M,\g)$ such that for all $\tau\in U$, $A_{\tau}$ is a flat connection whose holonomy is in the class of $\tau$.
\end{proposition}
\begin{proof}
With our assumption, the quotient map $\boR(M)\to \boM(M)$ is a fibration over $U$ as $U$ consists in regular representations. As $U$ is contractible, there is a smooth section $\rho:U\to \boR(M)$ of the quotient map.
Fix a point $x$ in $M$ and consider the product $P=\tilde{M}\times \su\times U/\sim$ where we set $(\gamma.m,\rho_{\tau}(\gamma)g,\tau)\sim (m,g,\tau)$ for all $\gamma\in \pi_1(M,x)$. This construction gives a flat $G$-bundle over $M\times U$ such that for each $\tau\in U$, the holonomy of $P$ over $M\times\{\tau\}$ is given by $\rho_{\tau}$.

Moreover, $M\times U$ is homotopic to $M$ and $P$ has to be trivial as a $G$-bundle. Taking a smooth section $s:M\times U\to P$, we pull back the flat structure of $P$ to a flat connection $A$ on $M\times U$. The restriction of $A$ to each slice $A\times\{\tau\}$ gives the connection $A_{\tau}$ that we are looking for.
\end{proof}

As a first application of gauge theory, we finally prove that $\boM^{\reg}(\Sigma)$ is a symplectic manifold.
\begin{proposition}
Let $\Sigma$ be a closed surface. The non-degenerate 2-form $\omega$ on the regular part of $\boM(\Sigma)$ is closed, and hence symplectic.
\end{proposition}
\begin{proof}
Let us show it on any open set $U$ as in Proposition \ref{releve}. Let $A:U\to\Omega^1(M,\g)$ be the map given by this proposition. The remark following Theorem \ref{derham} implies that the derivative of $A$ at $\tau$ is a map $T_{\tau}\boM(\Sigma)\to\Omega^1(M,\g)$ taking values in $\ker d_{A_{\tau}}$. Considering its class in $H^1_{A_{\tau}}(\Sigma,\g)$, one gets the inverse De Rham isomorphism. We conclude that the form $\omega$ on $U$ is the pull-back of the form $\omega(\alpha,\beta)=\int_{\Sigma}\langle \alpha\wedge\beta\rangle$ on $\Omega^1(\Sigma,\g)$.
This latter expression is a constant 2-form on an infinite dimensional space. It is closed in the sense that for any smooth maps $X,Y,Z$ from $\Omega^1(\Sigma,\g)$ to itself, the following identity holds: 
\begin{equation*}
\begin{split}
''d\omega(X,Y,Z)''=X.\omega(Y,Z)-Y.\omega(X,Z)+Z.\omega(X,Y)+\omega(X,[Y,Z])+\\
\omega(Y,[Z,X])+\omega(Z,[X,Y])=0
\end{split}
\end{equation*}
This identity, pulled back to $U$ implies that $\omega$ is closed.

\end{proof}

\section{Chern-Simons theory}

The term Chern-Simons theory usually consists in the study of secondary characteristic classes on flat bundles. Indeed, by Chern-Weyl theory, we know that given $\pi:P\to M$ a $G$-bundle, we can compute the characteristic classes (Chern, Euler and Pontryagin classes) of associated bundles by integrating invariant polynomials in the curvature of some connection of $P$. The existence of flat connections implies the vanishing of all characteristic classes. Chern and Simons introduced some primitives of the Chern-Weyl classes giving non trivial invariants of flat $G$-bundles. For our purposes, we will reduce Chern-Simons theory to the following constructions:

\begin{enumerate}
\item Given a closed 3-manifold $M$, we construct a locally constant map $CS:\boM(M)\to \R/4\pi^2 \Z$.
\item Given a closed surface $\Sigma$, we get an hermitian line bundle with connection $(\boL,|\cdot |,\nabla)$ over the regular part of $\boM(\Sigma)$ such that the curvature of $\nabla$ is the symplectic form $\omega$. We will call this bundle the prequantum bundle of $\Sigma$ and denote it by $\boL_{\Sigma}$.
\item Given a 3-manifold $M$ with boundary, we obtain a flat lift $CS$ of the restriction map $\boM(M)\to\boM(\partial M)$ to the prequantum bundle of $\partial M$.
$$\xymatrix{
& \boL_{\partial M}\ar[d]\\
\boM^{\rm reg}(M)\ar[r]\ar@{-->}^{CS}[ur]&\boM^{\rm reg}(\partial M)}$$
\end{enumerate}
\subsection{The Chern-Simons functionnal}

Let $M$ be a 3-manifold possibly with boundary, $\pi:P\to M$ a trivializable $G$-bundle and $\boF$ a flat structure on $P$.
Given a section $s:M\to P$, one obtains a flat connection $A$ by the procedure described in the last chapter. One set $CS(A)=\frac{1}{12}\int_{M} \langle A\wedge [A\wedge A]\rangle$.

Recall that $A$ belongs to $\Omega^1(M,\g)$ so that $A\wedge A\wedge A$ belongs to $\Omega^3(M,\g^{\otimes 3})$. Applying the antisymmetric map $(X,Y,Z)\mapsto \langle X,[Y,Z]\rangle$ to the coefficients, one obtains the 3-form $\langle A\wedge [A\wedge A]\rangle$ which can then be integrated.

The main point is to compute how this functional changes when changing the section. Given a map $g:M\to G$ one has
$$CS(A^g)=CS(A)+\frac{1}{2}\int_{\partial M} \langle g^{-1}Ag\wedge g^{-1}\dd g\rangle-\frac{1}{12}\int_{M}\langle g^{-1}\dd g\wedge [g^{-1}\dd g\wedge g^{-1}\dd g]\rangle.$$

The proof is a direct consequence of Stokes formula, with the use of some formulas for differential forms on Lie groups, see \cite{freed}.
Denote by $W(g)$ the term $\frac{1}{12}\int_{M}\langle g^{-1}\dd g\wedge [g^{-1}\dd g\wedge g^{-1}\dd g]\rangle$. It is also called Wess-Zumino-Witten functional. By definition, the form $g^{-1}\dd g$ is equal to $g^{*}\theta$ where $\theta$ is the left Maurer-Cartan form on $G$. Hence $W(g)=\int_M g^*\chi$ where $\chi=\frac{1}{12}\langle\theta\wedge[\theta\wedge\theta]\rangle$ is the Cartan 3-form on $G$. We deduce from this that assuming $G=\su$, $W(g) \mod 4\pi^2$ depends only on the restriction of $g$ to $\partial M$.

Indeed, given another 3-manifold $N$ with an oriented diffeomorphism $\phi:\partial N\to \partial M$ and a map $h:N\to G$ such that $g\circ \phi = h$, we consider the integral $\int_{M\cup(-N)}f^*\chi$ where $f$ stands for $g$ on $M$ and $h$ on $N$. This integral is equal to $W(g)-W(h)$. On the other hand, it is equal to $(\deg f) \int_G \chi= 4\pi^2\deg f \in 4\pi^2 \Z$.
This proves that if $M$ has no boundary, then $CS(A^g)=CS(A)\mod 4\pi^2$.

Hence, the map $CS: \boM(M,\su)\to \R/4\pi^2\Z$ is well-defined when $M$ has no boundary.

In order to show that it is locally constant, recall that if $A_t$ is a smooth family of flat connections with $A_0=A$ then the derivative $\alpha=\frac{\dd A_t}{\dd t}|_{t=0}$ satisfies $\dd \alpha+[\alpha\wedge A]=0$. Moreover, $\frac{\dd CS(A_t)}{\dd t}|_{t=0}=\frac{1}{4}\int_M\langle \alpha\wedge[A\wedge A]\rangle=-\frac{1}{2}\int_M \langle \alpha\wedge\dd A\rangle$.

On the other hand, $-\langle \alpha\wedge \dd A\rangle=\dd\langle\alpha\wedge A\rangle+\langle\dd\alpha\wedge A\rangle$. Moreover $\langle \dd\alpha\wedge A\rangle=-\langle [\alpha\wedge A]\wedge A\rangle=\langle [A\wedge A]\wedge \alpha\rangle=2\langle \dd A\wedge \alpha\rangle$.
This implies the identity $\langle \alpha\wedge \dd A\rangle=\dd\langle \alpha\wedge A\rangle.$ Hence, one has
$\frac{\dd CS(A_t)}{\dd t}|_{t=0}=\frac{1}{2}\int_{\partial M}\langle A\wedge \alpha\rangle$.

In the case where $M$ has no boundary, this proves that the Chern-Simons function is locally constant on $\boM(M)$.

\subsection{Construction of the prequantum bundle}\label{prequantum}

Let $\Sigma$ be a closed compact surface. Recall that for $G=\su$, all principal $G$-bundles are trivial, hence flat structures are encoded by flat connections $A\in \Omega^1_{\flat}(\Sigma,\g)$ (that is $A$ satisfies $\dd A+\frac{1}{2}[A\wedge A]=0$). Moreover, two connections represent the same element of $\boM(\Sigma)$ if and only they are related by the action of $g:\Sigma\to G$.
Consider the finer equivalence relation where $A$ and $A^g$ are considered to be equivalent if $CS(A)=CS(A^g)$. New equivalence classes form a bundle over the old ones with fiber $\R/4\pi^2\Z$. This is the construction of the prequantum bundle. Let us give another point of view of the same construction, technically more appropriate.

Set $L=\Omega^1_{\flat}(\Sigma,\g)\times \R/2\pi\Z$ and define an action of the gauge group on $L$ by the formula:
$(A,\theta)^g=(A^g,\theta+c(A,g))$ where 
$$c(A,g)=\frac{1}{4\pi}\int_{\Sigma} \langle g^{-1}Ag\wedge g^{-1}\dd g\rangle-\frac{1}{2\pi}W(g).$$
Then, $c$ is a cocycle in the sense that for any flat connection $A$ and gauge group elements $g,h$ one has $c(A,gh)=c(A,g)+c(A^g,h)$. 

Consider the quotient map $L/\Gamma\to \Omega^1_{\flat}(\Sigma,\g)=\boM(\Sigma)$. By using the local sections of the projection $\Omega^1_{\flat}(\Sigma,\g)\to \boM(\Sigma)$ given by Proposition \ref{releve} and the fact that the gauge group acts freely on connections encoding irreducible representations, one find that the above quotient is actually a principal fiber bundle over $\boM^{\rm reg}(M)$ with fiber $\R/2\pi\Z$.
The prequantum line bundle $\boL$ is the fiber bundle associated to $L/G$ with the representation of $\R/2\pi\Z$ on $\C$ given by $\theta.z=e^{i\theta}z$. It is naturally an hermitian line bundle.

Let us show that this bundle has a connection with curvature $\omega$.

On the trivial bundle $L\to \Omega^1_{\flat}(\Sigma,\g)$ there is a natural connection given by the expression $\dd-\lambda$ where $\lambda$ is the 1-form given by $\lambda_A(\alpha)=\frac{1}{4\pi}\int_{\Sigma}\langle A\wedge \alpha\rangle$.
One can check directly that this form is equivariant and hence defines a connection on the quotient $\boL$. The curvature of this connection is the derivative of $\lambda$, that is $\frac{1}{2\pi}\omega$. The same formula is true on the quotient as both symplectic forms $\omega$ on $\Omega^1_{\flat}(\Sigma,\g)$ and $\boM(\Sigma)$ correspond in the quotient.

We can give the third and last application: given a 3-manifold with boundary $M$, and a flat connection $A$ on it, we consider the element $(A,CS(A)/2\pi)\in L$. Given $g:M\to G$, the connection $A^g$ will be sent to the equivalent element $(A^g,CS(A^g)/2\pi)$. Hence, this map also denoted by $CS$ is well defined from $\boM(M)$ to $\boL_{\partial M}$.
Moreover, given a smooth family $A_t$, we already computed $\frac{\dd CS(A_t)}{\dd t}=(\alpha,\frac{1}{4\pi}\int_{\partial M}\langle A\wedge \alpha\rangle$), where $\alpha=\frac{\dd A_t}{\dd t}$. This derivative is in the kernel of the connection $\dd -\lambda$, which shows that $CS(A_t)$ is a parallel lift over $\boL_{\partial M}$ over the restriction of $A_t$ to $\partial M$ as asserted.

\subsection{Examples}
\subsubsection{Closed 3-manifolds}
Let us look at some examples of Chern-Simons invariant for a closed manifold $M$. Recall that we constructed a locally constant map $CS:\boM(M)\to \R/4\pi^2\Z$. The trivial representation is obtained as the holonomy of the connection $A=0$. In that case, one has $CS(A)=0$.

For less trivial examples, consider some manifolds obtained as a quotient of $S^3=\su$ by a finite subgroup $H$, for instance the lens spaces $L(p,1)$ given by $H_p=\{\begin{pmatrix} e^{2ik\pi/p}&0\\0&e^{-2ik\pi/p}\end{pmatrix}, k\in \Z/p\Z\}$, or the quaternionic manifold $Q_8$ given by $H=\{\pm 1, \pm {\bf i},\pm{\bf j},\pm{\bf k}\}$.
In these cases $M=\su/H$ and there is a natural non trivial flat bundle $P$ given by the quotient of $\su\times \su$ by the equivalence relation $(g_1,g_2)\sim (g_1 h,h^{-1}g_2)$ for $h\in H$. The map $\pi:P\to M$ is the first projection. A section is given by $s(g)=(g,g^{-1})$. One compute that the connection associated to that section is $g^{-1}\dd g$. Hence, $CS(A)=\int_{S^3/H}\chi=4\pi^2/|H|$.

Another easy example is $M=S^1\times S^1\times S^1$. In that case, all representations are abelian, hence all flat connections are equivalent to connections with values in $\R{\bf i}$. As $\langle {\bf i},[{\bf i},{\bf i}]\rangle=0$, one has necessarily $CS(A)=0$ for all $A$. This is compatible with the fact that $\boM(M)$ is connected and $CS$ is locally constant.

\subsubsection{The torus case}
Let us give a finite dimensional construction for the prequantum bundle $\boL$ over the torus $\Sigma=S^1\times S^1$.
Consider the map $F:\R^2\to \boM(\Sigma)$ sending $(\phi,\psi)$ to the representation $\rho_{\phi,\psi}:\Z^2\to \su$ where $\rho_{\phi,\psi}(a,b)=\exp({\bf i} (a\phi+b\psi))$.
The fibers of $F$ are the orbits of the action of the group $H=\Z^2\rtimes \Z/2\Z$ where the first factor acts by translation and the second one by inversion.

One can lift the map $F$ to $\Omega^1_{\flat}(\Sigma,\g)$ by sending $(\phi,\psi)$ to the connection $A_{\phi,\psi}={\bf i}(\phi\dd s+\psi\dd t)$ where $s,t$ are coordinates of the two $S^1$ factors identified to $\R/\Z$. 
The action of $H$ is realized by a gauge group action as follows. The translation by $(2\pi k, 2\pi l)$ is given by the action of $g_{k,l}(t,s)=\exp(2{\bf i}\pi (ks+lt))$, whereas the inversion is given by the action of the constant map $g={\bf j}$.

These actions lift to the trivial bundle $L=\R^2\times \R/2\pi\Z$ in the following way:
$(\phi,\psi,\theta)^{g_{k,l}}=(\phi+2\pi k,\psi+2\pi l,\theta+\frac{1}{4\pi}\int_{\Sigma}
\langle A_{\phi,\psi},g_{k,l}^{-1}\dd g_{k,l}\rangle)$ as one can show that $W(g_{k,l})=0$ (see \cite{kauf}).
We obtain $(\phi,\psi,\theta)^{g_{k,l}}=(\phi+2\pi k,\psi+2\pi l,\theta+\phi l-\psi k)$ and $(\phi,\psi,\theta)^{\bf j}=(-\phi,-\psi,\theta)$. We recognize an action of $H$ on $L$. The quotient produces a bundle over the quotient of $\R^2$ by $H$, smooth over non central representations. This gives an elementary construction of the prequantum bundle in that case which is very useful for computing Chern-Simons invariant for knot exteriors.

\subsubsection{Some knot complements}

Let us give an application of these constructions in the case of a knot complement. In that case, the manifold $M$ is the complement of a tubular neighborhood of a knots in $S^3$. Its boundary is identified with $S^1\times S^1$. The Lagrangian immersion $r:\boM^{\reg}(M)\to \boM^{\reg}(S^1\times S^1)$ is shown in Figure \ref{embed}. What kind of information can we extract from the existence of a lift $CS:\boM^{\reg}\to \boL_{S^1\times S^1}$?
If we have a closed loop $\gamma$ in $\boM^{\reg}$, then we showed that it lifts to $\boL_{S^1\times S^1}$. In other terms, the holonomy of $\boL$ along $r(\gamma)$ is trivial. This holonomy is easy to compute as $e^{iA/2\pi}$ where $A$ is the symplectic area enclosed by $r(\gamma)$: hence $A$ has to be an integral multiple of $4\pi^2$. We can check this in the two examples of Figure \ref{embed}: more generally, this shows for instance that $r(\gamma)$ cannot be a small oval.

\section{Surfaces with higher genus}\label{genus}
Till now, we did not say much about surfaces with positive genus although they are very important and interesting. Let $\Sigma$ be such a surface of genus $g$. Then, we showed that $\boM(\Sigma)$ splits into an irreducible and an abelian part and that the irreducible part is a smooth symplectic manifold of dimension 6g-6, with a prequantum bundle $\boL\to \boM(\Sigma)$. To answer simple questions about the topology of that space or its symplectic volume, we introduce a family of functions called trace functions which give to $\boM(\Sigma)$ the structure of an integrable system.

\subsection{Trace functions and flat connection along a curve}

Let $\gamma$ a 1-dimensional connected submanifold of $\Sigma$ which do not bound a disc. We will call $\gamma$ a {\it curve}.
We associate to $\gamma$ the map $h_{\gamma}:\boM(\Sigma)\to [0,\pi]$ by the formula $h_{\gamma}([\rho])=\ang \rho(\gamma)$. This gives a well-defined and continuous function on $\boM(\Sigma)$, smooth where it is different from 0 and $\pi$. We will compute the Hamiltonian vector field $X_{\gamma}$ associated to this map and compute its flow, showing that it is $4\pi$-periodic. We will give later an interpretation of this flow in terms of twisting of flat bundles on $\Sigma$ along $\gamma$.

Our way of understanding such constructions uses heavily a lemma on the normalization of a flat connection along a curve that we state here without proof.

\begin{lemma}\label{normalize}
Let $\Sigma$ be a compact oriented surface and $\Phi:S^1\times[0,1]$ be an orientable embedding. Set $\gamma=\Phi(S^1\times\{1/2\})$ and let $U$ be a contractible open set in $\boM^{\reg }(\Sigma)$.
Suppose that for all $\tau$ in $U$, the representation indexed by $\tau$ take non-central values on $\gamma$. Then there is a smooth map $A:U\to\Omega^1_{\flat}(\Sigma,\lu)$ and a smooth map $\xi:U\to\g$ such that
\begin{enumerate}
\item The connection $A_{\tau}$ represents $\tau$.
\item $\Phi^* A_{\tau}= \xi(\tau) \dd t$ where $t$ is the coordinate identifying $S^1$ to $\R/\Z$.
\end{enumerate}
\end{lemma}

We will say that a flat connection $A$ on $\Sigma$ is {\it normalized along} $\Sigma$ if there exists $\xi\in \g$ such that $\Phi^*A=\xi\dd t$. In that case, the holonomy of $A$ along $\gamma$ is equal to $\exp(-\xi)$ and one has $h_{\gamma}(A)=\frac{1}{\sqrt{2}}||\xi|| \mod \pi$.

The aim of this section is to identify the hamiltonian vector field of $h_{\gamma}$ that is, the vector field $X_{\gamma}$ on $\boM^{\reg}(\Sigma)$ such that $i_{X_{\gamma}}\omega=\dd h_{\gamma}$. We will give a De Rham lift of this vector field assuming that all connections are normalized along $\gamma$. This description will allow us to compute its flow and its lift to the prequantum bundle.

\begin{proposition}\label{flow}
In the settings of Lemma \ref{normalize}, a lift of $X_{\gamma}$ at $A$ normalized such that $\Phi^*A=\xi\dd t$ is given by the connection $\Phi_*(\frac{-\xi}{\sqrt{2}||\xi||}\phi(s)\dd s )$ where $s$ is the coordinate of $[0,1]$ and $\phi$ is a function with support in $[0,1]$ and integral 1.
\end{proposition}
\begin{proof}
Let $A$ be a normalized flat connection. One need to prove the equality $\omega(X_{\gamma},Y)=\dd h_{\gamma}(Y)$ for all $Y$ in the tangent space of $[A]$. Thanks to Lemma \ref{normalize}, we can normalize all connections in the neighborhood of $A$, hence one can suppose that all tangent connections $\alpha$ are such that $\Phi^*\alpha= \eta \dd t $ for some $\eta \in \lu$.
One has $\omega(X_{\gamma},\alpha)=\int_{S^1\times[0,1]} \langle \phi(s)\dd s\frac{-\xi}{\sqrt{2}||\xi||}\wedge \eta \dd t\rangle=\frac{\langle \xi,\eta\rangle}{\sqrt{2}||\xi||}$.
On the other hand, $\dd h_{\gamma} (\alpha)=\frac{\langle \xi,\eta\rangle}{\sqrt{2}||\xi||}$. This proves the formula.
\end{proof}

The same proof gives that the hamiltonian flow of $h_{\gamma}$ sends the normalized connection $A$ to $\Phi_{\gamma}^T(A)=A+TX_{\gamma}$. This shows in particular that this flow is periodic as $A+4\pi X_{\gamma}=A^g$ for
$$g(t,s)=\exp(-\int_0^s\phi(u)\dd u \frac{2\sqrt{2}\pi \xi}{||\xi||}).$$

In this formula, $g$ is a smooth function which is equal to 1 outside the image of $\Phi$. In the case where $\gamma$ is separating, one can replace $4\pi$ by $2\pi$ and the function $g$ will still be well-defined, being equal to 1 on one side and to -1 on the other side. This shows that the hamiltonian flow of $h_{\gamma}$ for separating curves is $2\pi$-periodic although it is $4\pi$-periodic for non separating curves. Let us give a geometric interpretation of these flows.

Let $(P,\boF)$ be a flat $\su$-bundle over $\Sigma$ and $\gamma$ be a curve on $\Sigma$. The holonomy of $\boF$ along $\gamma$ is a transformation of the fiber which we suppose to be non-central.

Cutting $\Sigma$  on $\gamma$, we get a new closed surface $\hat{\Sigma}$ with two circles at the boundary. Let $\nu:\hat{\Sigma}\to\Sigma$ be the gluing map. 
The flow $\Phi_{\gamma}^t(P,\boF)$ is obtained as a quotient of the form $\nu^*(P,\boF)/\sim_t$.

To give a precise formula for $\sim_t$, we orient $\gamma$ and take a point $p$ on $\pi^{-1}(\gamma)$. The holonomy along $\gamma$ in the positive direction sends $p$ to $p.g$. Let $\gamma^+$ (resp. $\gamma^-$) be the component of $\partial\hat{\Sigma}$ which respect (resp. do not respect) the orientation of $\gamma$. Let $p^+$, $p^-$ be the preimages of $p$ in the corresponding fibers. Then,  by  definition $p^+.\exp(\frac{t\xi}{\sqrt{2}||\xi||})\sim_t p^-$ where $\xi$ is the unique element of $\lu$ with $||\xi||<2\sqrt{2}\pi$ such that $g=\exp(\xi)$.
We claim that there is a unique isomorphism of flat $\su$-bundles $\sim_t:\pi^{-1}(\gamma^+)\to \pi^{-1}(\gamma^-)$ which extends the previous formula.

We obtain this description easily from the previous one by integrating the connection $A$ in directions transverse to $\gamma$.

\subsection{Global description of the moduli space}\label{integrable}

Given a closed surface of genus $g>1$, the maximal number of disjoint curve is $3g-3$.
Let $(\gamma_i)_{i\in I}$ be such a family. It decomposes the surface into pairs of pants in the sense that the complement of the curves $\gamma_i$ is a disjoint union of $2g-2$ discs with two holes. It is convenient to construct from this decomposition a trivalent graph $\Gamma$. The set of vertices denoted by $V(\Gamma)$ corresponds to pair of pants and edges to curves. An edge is incident to a vertex if the corresponding curve bounds the corresponding pair of pants. An example is shown in Figure \ref{bretzel} for $g=2$.

\begin{figure}[h]\label{bretzel}
\begin{center}
\includegraphics[width=10cm]{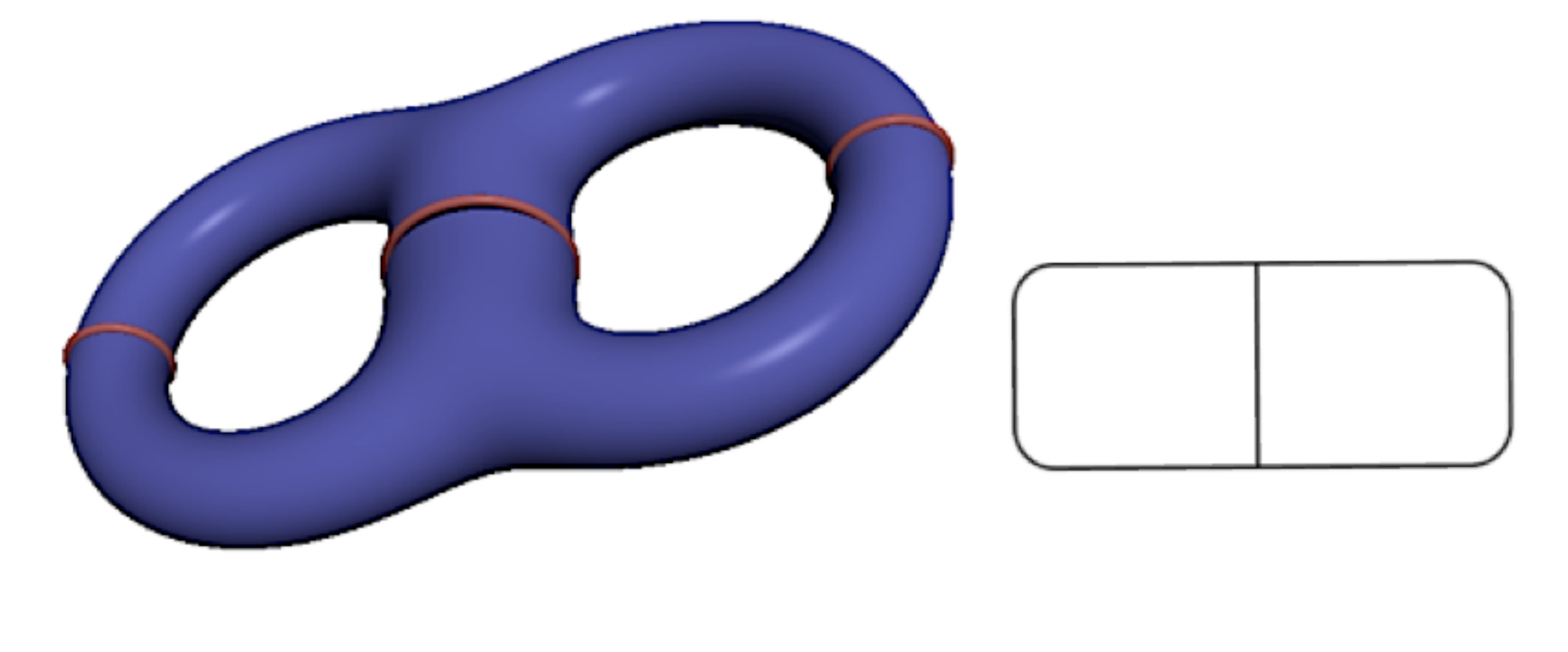}
\caption{Pants decomposition of a genus 2 surface}
\end{center}
\end{figure}

Consider the map $h:\boM(\Sigma)\to [0,\pi]^{I}$ given by $\rho\mapsto (h_{\gamma_i}(\rho))_{i\in I}$. This is an integrable system in the sense that it is a maximal set of Poisson commuting functions. We sum up the properties of this map in the following proposition.

\begin{theorem}
The image of $h$ is the polyhedron $\Delta$ consisting of the $(\alpha_i)_{i\in I}$ such that for any trivalent vertex $v$ of $\Gamma$ the following relation holds:
$$|\alpha_i-\alpha_j|\le \alpha_k\le \min(\alpha_i+\alpha_j,2\pi-\alpha_i-\alpha_j)$$
if $i,j,k$ are incident to $v\in V(\Gamma)$. 
Set $\boM^{\circ}(\Sigma)=h^{-1}( (0,\pi)^I)$. Then, the flows $\Phi_{\gamma_i}^t$ commute on $\boM^{\circ}(\Sigma)$ and cover the fibers of $h$.
\end{theorem}
\begin{proof}
Let $v$ be a vertex of $\Gamma$ and $\rho$ an element of $\boM(\Sigma)$. Then, restricting $\rho$ to the pair of pants $P_v$ encoded by $v\in V(\Gamma)$, we obtain an element of $\boM(P_v)$. Let $i,j,k$ be the edges incident to $v$: then $\alpha_i,\alpha_j,\alpha_k$ are the angles of $\rho$ on the boundary components of $P_v$. Proposition \ref{pants} tells us that they need to satisfy the inequalities of the theorem. This explains why the image of $h$ is in $\Delta$. Let us give a short explanation of the remaining part. Consider an element $\alpha$ in $\Delta$. Then, we know that there exist corresponding representations in the pants $P_v$ for all $v$, unique up to conjugacy. One can realize these representations as flat bundles. For each edge in $\Gamma$, we glue the corresponding boundary curves. The holonomy of the flat bundles are conjugated, proving that one can glue them into a flat bundle on $\Sigma$. This proves the surjectivity of $h$.
In the case all these holonomies are non central, then all possible ways of gluing these bundles are described by the Hamiltonian flow of the corresponding angle function. By construction these flows commute and cover the fibers of $h$.
\end{proof}

Let us describe more precisely the fiber of $h$. Fix an element $\rho$ of $\boM^{\circ}(\Sigma)$ and set $\alpha=h(\rho)$. The joint Hamiltonian flow of the functions $(h_{\gamma_i})$ give an action of $\R^I$ on the fibre $h^{-1}(\alpha)$ by the formula $t.\rho=\Phi_{\gamma_1}^{t_1}\cdots \Phi_{\gamma_N}^{t_N}(\rho)$ for any numbering of the elements of $I$. The kernel of this action is precisely described by the lattice $4\pi\Lambda\subset\R^I$ where we set:

$$\Lambda=\text{Vect}_{\Z}\{e_i, e_v, i\in I, v\in V(\Gamma)\}$$
In this formula $e_i$ is the basis element with coordinate $(\delta_{ij})$ whereas we set $e_v=(e_i+e_j+e_k)/2$ where $i,j,k$ are the edges incident to $v$.

We already showed  that the flows $\Phi_{\gamma}$ are $4\pi$-periodic, which explains why $e_i$ belongs to $\Lambda$ for all $i$. We prove in the same way that $4\pi e_v$ is in the kernel of the action: let $\gamma_i, \gamma_j,\gamma_k$ be three curves bounding $P_v$. Let $A$ be a flat connection representing $\rho$, normalized in the neighborhood of the three curves. Then, after applying the three flows during a time $2\pi$, we obtain a gauge equivalent connection where the gauge element is equal to -1 in the interior of $P_v$, 1 in the exterior and is given in the standard neighborhoods of the three curves by the same formulas as in Proposition \ref{flow}. This shows that $4\pi\Lambda$ belongs to the kernel of the action, we refer to \cite{jw3} for the proof that these lattices are actually equal.
\subsection{Some applications}
There are plenty of applications as this description is very precise for geometric and symplectic aspects. As an example, let us describe the symplectic structure in this setting and give a formula for the symplectic volume of $\boM(\Sigma)$.

\begin{proposition}
Let $s:\inter(\Delta)\to\boM^{\circ}$ be a Lagrangian section of $h$ over the interior of the polyhedron $\Delta$. The map $\Phi:\inter(\Delta)\times\R^I/\Lambda\to \boM^{\circ}$ defined by $\Phi(\alpha,t)=t.s(\alpha)$ is a diffeomorphism on $h^{-1}(\inter(\Delta))$ and we have
$\Phi^*\omega=\sum_i \dd\alpha_i\wedge \dd t_i$.
\end{proposition}
One deduce from this formula that the volume of $\boM(\Sigma)$ is equal to the volume of the dense open subset $h^{-1}(\inter(\Delta))$ which is equal to $\vol(\R^I/4\pi\Lambda)\vol(\Delta)$.

Moreover, we have $\vol(\R^I/4\pi\Lambda)=\vol(\R^I/4\pi\Z^I)/[\Lambda,\Z^I]$ which is finally equal to $(4\pi)^{3g-3}/[\Lambda,\Z^I]$.
To compute the index of $\Z^I$ in $\Lambda$ we notice that it is equal to the dimension of $C^1(\Gamma,\Z_2)$ divided by coboundaries. We find $[\Lambda,\Z^I]=\dim H^1(\Gamma,\Z_2)=2^g$. Hence
$$\vol(\boM(\Sigma))=(2\pi)^{3g-3} 2^{2g-3}\vol(\Delta).$$

\section{Introduction to geometric quantization}
In this section, we introduce some basic objects of geometric quantization. It is a procedure which associates to a symplectic manifold $M$ with extra structure a vector space $Q(M)$ called "quantization of $M$". By construction, some functions on $M$ act on $Q(M)$ with commutation relations prescribed by the Poisson bracket. Of course a good example to keep in mind is $T^*\R^n$ whose quantization is $L^2(\R^n)$ and where position, momentum and Hamiltonian operators are quantization of the coordinates and the energy. In full generality, our construction is naive and not well motivated but produces at least vector spaces and operators. We compute them in the case of moduli space and describe the so-called Bohr-Sommerfeld leaves which coincide with the spectrum of curve operators in Chern-Simons topological quantum field theory as initiated by Witten in \cite{witten}. A detailed introduction to geometric quantization can be found in \cite{gs,weinstein}.
The computation of Bohr-Sommerfeld fibers was done in \cite{jw1}. In these notes, we obtain it in a more direct way and take into account the "metaplectic correction".

\subsection{Spin structures}\label{spin}
Recall that for all $n$, there is a group Pin$(n)$ sitting in the following exact sequence
$$1\to\Z_2 \to \text{Pin}(n)\to \text{O}(n)\to 1.$$
This extension is caracterized by the two following features: over SO$(n)$, it is the unique 2-fold covering (universal for $n>2$) and any lift in Pin$(n)$ of a reflection in 0$(n)$ has order 2.
As GL$(n)$ retracts on O$(n)$ there is a unique group $\tgl(n)$ sitting in the exact sequence
$$1\to\Z_2 \to \tgl(n)\to \text{GL}(n)\to 1.$$
and which is isomorphic to Pin$(n)$ when restricted to O$(n)$

On a vector space $V$ of dimension $n$, we denote by $\boR(V)$ the set of basis of $V$. It is an homogeneous space over GL$(n)$. We call spin structure on $V$ a set $\tilde{\boR}(V)$ with a free transitive action of $\tgl(n)$ and a map $p:\tilde{\boR}(V)\to\boR(V)$ intertwining the actions of $\tgl(n)$ and GL$(n)$.

Spin structures on $V$ form a category Sp$(V)$ where all objects are isomorphic with precisely two isomorphisms. This category is equivalent to the category with one object and automorphism group $\Z_2$.

Given two vector spaces $V,W$, there is a functor $F$ from $\sp(V)\times\sp(W)$ to $\sp(V\oplus W)$ sending $(\tilde{\boR}(V),\tilde{\boR}(W))$ to $\tgl(n+m)\times\tilde{\boR}(V)\times\tilde{\boR}(W)/\sim$. The equivalence relation is generated by $(hg,hs_v,s_w)\sim (g,s_v,s_w)$ for $h\in \tgl(n)$ and  $(hg,s_v,hs_w)\sim (g,s_v,s_w)$ for $h\in \tgl(m)$.
This functor $F$ is equivalent to the functor trivial on objects and sending $\Z_2\times \Z_2$ to $\Z_2$ via the addition.

Let us give two generalizations of this construction: in the first one, we consider an exact sequence $0\to U\to V\to W\to 0$. Choosing a section $s:W\to V$ gives an isomorphism from $U\oplus W$ to $V$ sending $(u,w)$ to $u+s(w)$, and hence as before a functor $F_s:\sp(U)\times\sp(W)\to \sp(V)$. This functor depends on $s$ but only up to a unique natural transformation, so this dependance is not relevant for categorical purposes.

Our last generalization consists in a comparison between spin structures on a complex and on its cohomology, equivalent to Lemma \ref{lem-torsion}.
\begin{lemma} \label{lem-spin}
Given a finite dimensional complex $C^*=C^0\to \cdots \to C^n$, we define $\sp(C^*)=\sp(\bigoplus_i C^i)$, and $\sp(H^*)=\sp(\bigoplus_i H^i(C^*))$. There is an equivalence of categories $\sp(C^*)\to \sp(H^*)$ well-defined up to natural transformation.
\end{lemma}
\begin{proof}
In the settings of the proof of Lemma \ref{lem-torsion}, choose objects in $\sp(H^i)$ and $\sp(B^i)$. The exact sequence $0\to B^{i-1}\to Z^i\to H^i\to 0$ give a functor from $\sp(B^{i-1})\times\sp(H^i)$ to $\sp(Z^i)$ whereas the exact sequence $0\to Z^i\to C^i\to B^i\to 0$ gives a functor from $\sp(Z^i)\times\sp(B^i)$ to $\sp(C^i)$.
One obtains finally an element of $\prod_i \sp(C^i)$ which we send to $\sp(\bigoplus_i C^i)$.
Applying a non trivial automorphism on the element of $\sp(B^i)$ do not change the result as this element appears twice in the result. This shows the lemma as all functors are well-defined up to natural transformation.
\end{proof}

\begin{definition}
A spin structure on a manifold $M$ of dimension $n$ is a left $\tgl(n)$-principal bundle $\tilde{\boR}(M)$ on $M$ with a bundle map $\pi:\tilde{\boR}(M)\to\boR(M)$, the GL$(n)$-bundle of framings on $M$ which intertwines the actions of $\tgl(n)$ and GL$(n)$.
\end{definition}
In short, it is a smooth collection of spin structures for all tangent spaces of $M$. Two spin structures are isomorphic if there is an isomorphism of $\tgl(n)$-bundles commuting with the projections on the framing bundle. We denote by $\sp(M)$ the category of spin structures on a manifold $M$.

The simplest example is the circle, on which there are two isomorphism classes of spin structure that we obtain in the following way.
\begin{enumerate}
\item Identify all tangent spaces of $\R/\Z$ with $\R$ and take the same spin structure on this "constant" tangent space.
\item Consider $S^1\subset \R^2$ and consider the induced spin structure, using the trivial spin structure in $\R^2$ as above and the equivalence $\sp(TS^1)\times \sp(N)\simeq \sp(\R^2|_{S^1})$ where $N$ is the (trivial) normal bundle of $S^1$ in $\R^2$.
\end{enumerate}
In the first case, the bundle $\tilde{\boR}(S^1)$ is a trivial covering of $\boR(S^1)$ whereas the covering is non trivial in the second case. Much more generally, a manifold $M$ admits a spin structure if and  only if its second Stiefel-Whitney class $w_2(M)$ vanishes and in that case, isomorphism classes of spin structures form an affine space directed by $H^1(M,\Z_2)$.

In the case of moduli spaces of closed surfaces, Lemma \ref{lem-spin} gives us the following proposition.

\begin{proposition}
Let $\Sigma$ be a closed surface. There is a well-defined spin structure on $\boM^{\reg}(\Sigma,\su)$. More precisely, we define a functor $$F:\sp(H^*(\Sigma,\R)\otimes \g)\to \sp(\boM^{\reg}(\Sigma,\su).$$
\end{proposition}

\begin{proof}
This is a direct consequence of Lemma \ref{lem-spin} as we have an equivalence of categories
$\sp(H^*(\Sigma,\Ad_{\rho}))\sim \sp(C^*(\Sigma,\Ad_{\rho}))$. The second category do not depend on $\rho$ since $\rho$ appears only in the differnetials. Hence, all these categories are identified together. In the case where $\rho$ is irreducible, the first category reduces to $\sp(H^1(\Sigma,\Ad_{\rho}))\sim \sp(T_{[\rho]}\boM(\Sigma))$. On the contrary, if $\rho$ is the trivial representation, the first category reduces to $\sp(H^*(\Sigma,\g))=\sp(H^*(\Sigma,\R)\otimes \g)$ hence proving the proposition.
\end{proof}

\subsection{Lagrangian foliations}

The main ingredient in the geometric quantization of symplectic manifolds $(M,\omega)$ (for real  polarizations) is a Lagrangian foliation, that is an integrable distribution of Lagrangian subspaces of $TM$. For our purposes,we will suppose that this foliation may have singularities and that its leaves are the regular fibers of a map $\pi:M\to B$ which is a submersion over a dense open subset of $B$. Asking that the fibers are Lagrangian is equivalent to asking that all functions on $M$ written as $f\circ\pi$ for $f:B\to \R$ Poisson commute.

Let us give a list of classical examples to keep in mind:
\begin{enumerate}
\item
Let $V$ be a symplectic vector (or affine) space, then any linear Lagrangian subspace $L$ of $V$ gives such a foliation and the fibration is given by the quotient $V\to V/L$.
\item
Given a manifold $M$, its cotangent space $T^*M$ is a symplectic manifold foliated by the individual cotangent spaces. The fibration we are looking at is the natural projection $\pi:T^*M\to M$. This model corresponds physically to the canonical quantization of $M$.
\item
Let $(V,\omega,q)$ be a symplectic plane with a positive quadratic form $q$. The level sets of $q$ give a foliation of $V$ minus its origin. The map $q:V\to \R^+$ is our desired fibration with $0$ as a singular fibre. This model is referred to as the harmonic oscillator.
\item
Let $(E,q)$ be an oriented 3-dimensional euclidian space. Let $\alpha$ be a positive number and $S=q^{-1}(\alpha^2)$ be the sphere of radius $\alpha$. Then the sphere $S$ inherits a symplectic structure where symplectic frames at $x$ are by definition couples $(v,w)$ tangent to $S$ such that the determinant of $(v,w,x)$ is $\alpha^2$. Any linear form $\lambda$ on $E$ restricts to a lagrangian fibration $S\to \R$ with two singularities, its maximum and minimum.
\end{enumerate}
The example we are interested in is $\boM(\Sigma,\su)$. Given a pants decomposition of $\Sigma$ with cutting curves $(\gamma_i)_{i\in I}$, we get a map $h:\boM(\Sigma)\to [0,\pi]^I$. We showed in the last section that this map takes its values in a polyhedron $\Delta\subset[0,\pi]^I$ and that it is a Lagrangian fibration over the interior of $\Delta$.

\subsection{Bohr-Sommerfeld leaves}
Let $(M,\omega)$ be a symplectic manifold. We will ask that there is a prequantum bundle $\boL$ over $M$, a spin structure $s\in \sp(M)$ and a lagrangian fibration $\pi:M\to B$ maybe with singularities. The first condition is a symplectic one: a prequantum bundle exists if and only if $\omega/2\pi$ is an integral class in $H^2(M,\R)$. The second is topological as a spin structure exists iff $w_2(M)=0$. In both cases, there is no unicity unless $H^1(M,\Z)=0$.
The geometric quantization process gives a way for constructing a Hilbert space from these data which is finite dimensional if $M$ is compact and on which functions factorizing with the projection $\pi:M\to B$ have a natural quantization as operators. Moreover, quantizations coming from different lagrangian foliations can be compared by using a pairing introduced by Blattner, Kostant and Sternberg.

\subsubsection{Half-form bundle and quantization}
Let $L$ be a Lagrangian submanifold of a symplectic manifold $(M,\omega)$ with spin structure $s$. We will define in this section a bundle $\det^{1/2} L$, square root of the line bundle of volume elements on $L$. We start with some preliminaries.

Consider on $\R^{2n}$ the symplectic form $\sum_i \dd x_i\wedge \dd y_i$ where we used standard coordinates $(x_1,\ldots,x_n,y_1,\ldots,y_n)$. Any basis $(e_1,\ldots,e_n)$ of $\R^n$ can be extended to a unique symplectic basis $(e_1,\ldots,e_n,f_1,\ldots,f_n)$ of $\R^{2n}$ adapted to the decomposition $\R^{2n}=\R^n\oplus\R^n$. This gives a map from GL$(n)$ to GL$(2n)$. Pulling back the bundle $\tgl(2n)\to$GL$(2n)$, we get a new group $\hat{\text{GL}}(n)$ with a projection to GL$(n)$. This group is called the metalinear group and has the following down to earth description: it is isomorphic as a group to the direct product GL$^+(n)\times \Z_4$ with the projection to GL$(n)$ given by the map $(A,x)\to (-1)^x A$. There is an important morphism $\det^{1/2}:\hat{\text{GL}}(n)\to\C$ defined by $\det^{1/2}(A,x)=\sqrt{\det(A)}i^x$. Its square is the pull-back of the usual map $\det:\text{GL}(n)\to\R$.

Let $V$ be any symplectic vector space $V$ with a spin structure $\tilde{\boR}(V)$ and suppose it is symplectomorphic to $L\oplus L^*$ (such a symplectomorphism is equally described by a pair of transverse Lagrangians $L$ and $L'$, where $L'$ is the image of $L^*$). We can define a complex line $\det^{1/2}(L)$ in the following way. Pulling-back the spin structure of $V$ to $L$ as above, we get a metalinear structure on $L$ i.e a set $\hat{\boR}(L)$ homogeneous under $\hat{\text{GL}}(n)$. We define $\det^{1/2}(L)=\hat{\boR}(L)\times \C/\sim$ where $(hs,z)\sim(s,\det^{1/2}(h)z)$ for any $h$ in $\hat{\text{GL}}(n)$. As expected we have a natural isomorphism $\det^{1/2}(L)^{\otimes 2}\simeq \det(L)$.

We are ready to define the half-form bundle of a Lagrangian submanifold in a spin symplectic (called metaplectic) manifold $(M,\omega,s)$. Let $L$ be such a submanifold, and $L'$ be a Lagrangian subbundle of $TM|_L$ transverse to $TL$. Topologically, there is neither obstruction nor choice to find such a subbundle. Applying the preceding construction to all tangent spaces $TL_l$ for $l\in L$ give the line bundle $\det^{1/2}(L)$ that we are looking for. We will often denote it by $\delta$ for short.

In the case where $L$ is a regular fibre of a lagrangian fibration $\pi:M\to B$, there is an isomorphism between the cotangent space $T^*L$ at any point $l\in L$ and the tangent space $TB$ at $\pi(l)$. In other words, the tangent and cotangent spaces of lagrangian fibres are naturally trivialized. This isomorphism comes from the identification of $T^*L$ with the normal bundle $NL$ via $\omega$ and the derivative of $\pi$. Thanks to this isomorphim, one can give a flat structure to $TL$ and to all its associated bundles like $\det(L)$ or $\det^{1/2}(L)$. This explains why the half-form bundle $\delta$ is indeed a flat hermitian bundle on $L$. The underlying connection on it is often called the Bott connection.

\begin{definition}
Let $(M,\omega,\boL,s)$ be an enriched symplectic manifold and let $\pi:M\to B$ be a Lagrangian fibration. Then we will say that a fibre $L$ of $\pi$ is a Bohr-Sommerfeld fibre if there are non-trivial covariant flat sections of $\boL\otimes \delta$ over $L$. The space of all these sections will be denoted by $\boH(L,\boL\otimes\delta)$.
\end{definition}
If $L$ is connected, then $\boH=\boH(L,\boL\otimes\delta)$ is 1-dimensional. If moreover $L$ is compact, then there is an hermitian structure on the dual $\boH^*=\boH(L,\boL^*\otimes \delta^*)$ inducing one on $\boH$. Indeed, if $\lambda$ is a section of $\delta^*=\det^{-1/2}(L)$ then $\lambda\overline{\lambda}$ is a non negative density on $L$. We set $|s\otimes \lambda|^2=\int_L |s|^2 \lambda\overline{\lambda}$.

 Let $BS$ be the subset of $B$ parametrizing Bohr-Sommerfeld fibres. Then to each $b$ in $BS$ we have a complex line $\boH(\pi^{-1}(b),\boL\otimes \delta)$ (provided that $\pi^{-1}(b)$ is connected). We will denote by $\boH(M,\boL\otimes\delta)$ the space of sections of this line bundle over $BS$ - roughly speaking, this is the space of sections of $\boL\otimes\delta$ over $M$ covariantly constant along the fibres of $\pi$.
This rather ad hoc definition is motivated by the fact that $\boH(M,\boL\otimes\delta)$ has a natural inner product defined by $\langle s_1\otimes \lambda_1,s_2\otimes\lambda_2\rangle=\int_{BS}\langle s_1,s_2\rangle \lambda_1\overline{\lambda_2}$. In this formula, $BS$ is the subset of $B$ describing Bohr-Sommerfeld fibres supposed to be a codimension 0 submanifold of $B$ whereas $\lambda_1\overline{\lambda_2}$ is a density on $BS$ and then is ready to be integrated. There is a family of possibilities from compact fibres and discrete Boh-Sommerfeld set to non compact fibres and codimension 0 Bohr-Sommerfeld set. This is best understood with examples.

\subsubsection{Standard examples}
\begin{enumerate}
\item
If $(V,\omega)$ is a symplectic vector space, we construct a prequantum bundle $\boL$ by taking the trivial bundle $V\times \C$ with connection $d-\lambda$ where $\lambda$ is a 1-form on $V$ such that $\dd \lambda=\frac{1}{2\pi}\omega$.
Let $L$ be a linear lagrangian in $V$ and take a spin structure $\tilde{\boR}(V)$ on $V$. For auxilliary purposes, we also choose a Lagrangian $L'$ transverse to $L$. We explained in the last section that these data produce a complex line $\det^{1/2}(L)$ which is isomorphic to $\det^{-1/2}(L')$.
 Moreover, sections of $\boL$ constant along $L$ are determined by their restriction to $L'$: the quantization procedure reduces then to sections of $\det^{-1/2}(L')$ as all fibres are Bohr-Sommerfeld. This space is called the intrinsic Hilbert space of $L'$ as for any section $\lambda$ of it, the quantity $\lambda\ba{\lambda}$ can be integrated and the space of integrable sections is an Hilbert space. Certainly, different choices of $L'$ give isomorphic spaces although the isomorphism has to be computed using parallel transport along $L$.

\item
If $M$ is any manifold, we can construct a prequantum bundle $\boL$ on $T^*M$ by taking the product $T^*M\times\C$ with connection $d-\frac{1}{2\pi}\lambda$ where $\lambda$ is the Liouville 1-form. A spin structure on $T^*M$ induces a metalinear structure on $M$ and finally a square root $\delta$ of the line bundle $\det(M)^*$. The quantization associated to it will be the set of sections of $\delta$ over $M$, that is the intrinsic Hilbert space of $M$ (again, all fibres are Bohr-Sommerfeld).

\item
In the case of the harmonic oscillator $(V,\omega,q)$, we consider any prequantum bundle $\boL$ over $V$ (constructed as before) and a fixed spin structure on $V$. The fibres are the circles $q^{-1}(\alpha)$ for $\alpha>0$. The holonomy of $\boL$ along this leaf is $e^{i A}$ where $A$ is the symplectic area enclosed by $q^{-1}(\alpha)$. The half-form bundle $\delta$ restricts to each leaf to a flat non trivial line bundle. Its holonomy is then $-1$. Finally, Bohr-Sommerfeld fibres correspond to the values $\alpha$ such that $-e^{iA}=1$ where $A=c \alpha$ and $c$ is the volume of the disc $q^{-1}([0,1])$. This forces $\alpha$ to be of the form $k/c$ where $k$ an odd integer, correspondingly to the spectrum of the harmonic oscillator.
\item

Finally, the case of the sphere will be the most interesting one for us as it is the only compact example. Recall that we saw the sphere as the level set $q^{-1}(\alpha^2)$ and our normalization is such that its symplectic volume is $2\pi\alpha$. Then, we will be able to find a prequantum bundle $\boL$ if and only if $\alpha$ is an integer. Let $\lambda:V\to \R$ be a (unit) linear form and choose a spin structure $s$ on $S$ (unique up to homotopy). Regular fibres of $\lambda$ are circles on which $\delta$ is as before a non trivial flat bundle. The holonomy of $\boL$ along a fiber $\lambda^{-1}(l)$ is $e^{iA}$ where $A$ is the symplectic area enclosed by the circle. The quantity $A/2\pi$ is an integer if and only if $l$ is an integer satisfying $|l|\le \alpha$ and $l=\alpha\mod 2$. Hence, Bohr-Sommerfeld orbits correspond to integers $l$ satisfying $|l|\le \alpha$ and $l\ne \alpha \mod 2$.

\end{enumerate}

\subsubsection{The case of moduli spaces}

Consider the settings of Section \ref{genus}. As before $\Sigma$ will denote a closed surface of genus $g$ and the aim of this section is to investigate the quantization of $\boM^{\reg}(\Sigma,\su)$ with its symplectic form $\omega$ and Chern-Simons bundle $\boL$. For more generality, we will consider an integer $K$ called the level and multiply the symplectic form by $K$. This new symplectic form admits $\boL^{\otimes K}$ as a prequantum bundle. The spin structure is the one defined in Section \ref{spin} and the lagrangian fibration comes from the map $h:\boM(\Sigma)\to[0,\pi]^I$ sending $[\rho]$ to $(h_{\gamma_i}(\rho)_{i\in I}$ where $I$ indexes a maximal system of cutting curves $(\gamma_i)_{i\in I}$.
We need now to determine which fibers are Bohr-Sommerfeld, and to do this we must compute the holonomy of $\boL^{\otimes K}$ and $\delta$ along the fibres which are tori of the form $\R^I/\Lambda$. Hence, it is sufficient to compute the holonomy of both bundles along the generators of $\Lambda$. We do this in the following two next propositions.

\begin{proposition}
Fix $i\in I$ and $(\alpha_i)$ in $\inter(\Delta)$. Pick $\rho\in \boM^{\circ}(\Sigma)$ such that $h(\rho)=(\alpha_i)$. Let $e_i$ and $e_v$ be the generators of $\Lambda$ described in Section \ref{integrable}.
\begin{itemize}
\item
The holonomy of $\boL^{\otimes K}$ along the path $e_i$ is $e^{-2iK\alpha_i}$.
\item
The holonomy along the path $e_v$ for edges $i,j,k$ incident to the same vertex $v$ is $e^{-iK(\alpha_i+\alpha_j+\alpha_k)}$.
\end{itemize}
\end{proposition}
\begin{proof}
Orient all curves $\gamma_i$ and choose a cylinder $\Psi_i:S^1\times [0,1]\to \Sigma$ around them. Let $A$ be a flat connection in $\Omega^1_{\flat}(\Sigma,\lu)$ representing $\rho$ and suppose that it normalized along each cylinder: in formulas there are vectors $\xi_i$ in $\lu$ such that $\Psi_i^*A=\xi_i \dd t$. We showed in Proposition \ref{flow} that the Hamiltonian flow of $h_i$ changes $A$ to $A_T=\Phi^T_{\gamma_i}A$ such that $\Psi_i^*A_T=\xi \dd t -\frac{T\xi}{\sqrt{2}|\xi|}\phi(s)\dd s$. Suppose that at time T the element of $\boL$ is represented by $(A_T,\theta_T)\in \Omega^1_{\flat}(\Sigma,\lu)\times \R/2\pi\Z$ with $\theta_0=0$. Then as in Section \ref{prequantum}, this path in $\boL$ is parallel if and only if $\lambda(\frac{\dd}{\dd T}(A_T,\theta_T))=\theta_T'-\frac{1}{4\pi}\int_{\Sigma}\langle A_T,A_T'\rangle=0$. This equation reduces to $\theta_T'=-\frac{|\xi|}{4\sqrt{2}\pi}$ and $\theta_T=-\frac{T|\xi|}{4\sqrt{2}\pi}$. A computation shows that $(A_{4\pi},\theta_{4\pi})$ is equivalent to $(A,-\sqrt{2}|\xi|)$. This proves the first result as $|\xi|=\sqrt{2}h_{\gamma_i}$.

To compute the second term, we only need to replace $4\pi$ with $2\pi$ and take into account the contributions of the three curves $\gamma_i,\gamma_j$ and $\gamma_k$. The pair $(A,0)$ is transported to $(A,-\frac{\sqrt{2}}{2}(|\xi_e|+|\xi_f|+|\xi_g|))$ which gives the second result of the proposition.
\end{proof}

It remains to compute the holonomy of the half-form bundle $\delta$ along the fibre which we do in the following proposition.
\begin{proposition}
Let $L$ be the distribution of Lagrangian subspaces in $T\boM^{\circ}(\Sigma)$ given by $L_{\rho}=\ker D_{\rho}h$. Denote by $\delta$ the bundle $\det^{1/2}(L)$ as before. Then the holonomy of $\delta$ is 1 along $e_i$ and $-1$ along $e_v$.
\end{proposition}
\begin{proof}
We prove this proposition by considering half-periods: suppose that two representations $\rho_1,\rho_2:\pi_1(\Sigma)\to \su$ are the same when composed with the projection $\su\to$SO$(3)$. Then, because $\Ad_{\rho_1}=\Ad_{\rho_2}$, we have $T_{[\rho_1]}\boM(\Sigma)=T_{[\rho_2]}\boM(\Sigma)$.
Moreover, if we start the flow $\Phi_{\gamma}$ from a representation $\rho$ during a time $2\pi$, we reach the representation $\gamma^\#\rho$ where $\gamma^\#$ is the Poincar\'e dual of $\gamma$ in $H^1(\Sigma,\Z_2)=\Hom(\pi_1(\Sigma),\{\pm 1\})$. This fact is a direct consequence of the expression of the flow in Section \ref{flow}.
Finally, the map $h$ satisfies $h(\gamma^\# \rho)=\pm h(\rho)$. This shows that the Lagrangian subbundles at $\rho$ and $\gamma^\#\rho$ correspond so that we are allowed to compute the holonomy of $\delta$ along that path. If we show that this holonomy is $-1$ we prove the proposition as $e_i$ is a composition of 2 such paths and $e_v$ is a composition of 3 of them.

As $\inter(\Delta)$ is convex, all fibers of $h$ are isotopic and the holonomy of $\delta$ do not depend on which fiber we consider. Let us do the computation in the following case: let $A_i,B_i$ be the standard generators of $\pi_1(\Sigma)$ satisfying the following relation with $n$ even:
$$A_1B_1A_1^{-1}B_1^{-1}\cdots A_nB_nA_n^{-1}B_n^{-1}=1.$$

Then we define an element $\rho$ of $\boM^{\circ}(\Sigma)$ by setting $\rho(A_i)=\bi$, $\rho(B_i)=\bj$. Suppose that $\gamma$ is the curve represented by $A_1$ and set $[\rho_t]=\Phi_{\gamma}^t[\rho]$. Then we set $\rho_t(B_1)=\bj e^{\bi t/2}$ and the other values are unchanged.

The cell decomposition of $\Sigma$ consists in one 0-cell, $2n$ 1-cells and 1 2-cell. We identify $C_0(\Sigma,\Ad_{\rho_t})$ and $C_2(\Sigma,\Ad_{\rho_t})$ with $\lu$ and $C_1(\Sigma,\Ad_{\rho_t})$ with $\lu^{2n}$. The differentials $\dd h_{\gamma_i}$ belong to $H_1(\Sigma,\Ad_{\rho})$ and generate $L_{\rho}^*$, the dual of the Lagrangian subspace we are interested in. This elements are represented as twisted cycles by $\gamma_i\otimes f_i$ were $f_i$ is a non zero element of $\lu$ fixed by $\Ad_{\rho(\gamma_i)}$. This shows that all of them can be represented in $C_1$ by vectors which do not depend on $t$. Let us denote them by $F_i$ for $i\in I$. In order to understand the metalinear structure on $L$, we need to extend this basis to a symplectic basis of $H^1(\Sigma)$. We obtain in that way vectors $G_i$ for $i$ in I. We can choose them arbitrarily for $t=0$ and modify them in the vicinity of $\gamma$ for $t>0$. By this procedure, we can suppose that only the first two projections of the vectors $G_i$ depend on $t$.

Considering a fixed basis of $C_2$, we push it to $C_1$ with the injective map $\partial_2$ and get 3 vectors $U_1,U_2,U_3$. At the same time, we consider three vectors $V_1,V_2,V_3$ in $C_1$ whose image by $\partial_1$ is a fixed basis of $C_0$. By construction we get a basis $(V_1,V_2,V_3,U_1,U_2,U_3,F_i,G_i)$ of $C_1$. This basis depends on $t$ and defines for $t\in [0,2\pi]$ a closed path in GL$(C_1)$. The holonomy we are looking for is the homotopy class of this path.

To compute this path, we use the decomposition of $C_1$ into blocks corresponding to $A_1,B_1$ on one side and the other generators on the other side.
First, we have $\partial_1 (\xi_i,\eta_i)=\sum_i \left( \Ad_{\rho_t(A_i)} \xi_i-\xi_i + \Ad_{\rho_t(B_i)}\eta_i-\eta_i \right)$. In particular, $\partial_1(\xi_1,\eta_1)=\bi\xi_1\bi^{-1}+\bj e^{\bi t/2} \eta_1 (\bj e^{\bi t/2})^{-1}$ so that $\partial_1(-\bj z/2,-\bi x/2)=\bi x+\bj z$ for $x\in \R$ and $z\in \C$ so that we can set $\tl{V}_1=(0,-\bi/2), \tl{V_2}=(-\bj/2,0), \tl{V}_3=(\bk/2,0)$.

On the other hand, we compute
$$\partial_2(\xi)=(\bi\xi \bi^{-1}-\bk e^{\bi t/2}\xi e^{-\bi t/2}\bk^{-1},k e^{\bi t/2}\xi e^{-\bi t/2} k^{-1}-\bj e^{\bi t/2}\xi e^{-\bi t/2} \bj^{-1},\ldots)$$ where the dots mean the remaining components which do not depend on $t$.
We obtain $U_1=\partial_2(\bi)=(2\bi,0,\ldots), U_2=\partial_2(\bj)=(-\bj(e^{\bi t}+1),-2\bj e^{\bi t},\ldots),U_3=\partial_2(\bk)=(\bk(1-e^{\bi t}),-2\bk e^{\bi t},\ldots)$.

Writing the matrix of this basis we find a matrix $M(t)$ with the property that its derivative is block triangular with on the diagonal zeros and one 2*2 rotation matrix $R_t$ of angle $t$. This proves that in GL$(C_1)$, this path is not topologically trivial as the inclusion of $\text{GL}_2 \subset \text{GL}_n$ at the level of fondamental groups sends $R_t$ to the generator. This ends the proof in our case. One can show it for odd genus and separating curve $\gamma$ by considering other examples which are similar.
\end{proof}

Putting these results together, we find that Bohr-Sommerfeld orbits are parametrized by tuples $(\frac{\sigma_i \pi}{K})_{i\in I}$ where $\sigma_i$ are integers belonging to $[1,K-1]$ and satisfy the conditions that for all triples $i,j,k$ of adjacent edges,
\begin{enumerate}
\item $\sigma_i\le \sigma_j+\sigma_k$
\item $\sigma_i+\sigma_j+\sigma_k$ is odd.
\item $\sigma_i+\sigma_j+\sigma_k\le 2K$.
\end{enumerate}
These conditions are referred to as quantum Clebsch-Gordan conditions and describe a basis of the quantization of $\boM(\Sigma)$ as constructed for instance in \cite{bhmv}.

\subsection{Going further}

At this point, we defined the geometric quantization of the moduli space $\boM(\Sigma)$ as a finite dimensional Hilbert space depending on a Lagragian fibration, itself depending on a pants decomposition of $\Sigma$. The theory can be developped in the following directions:

\begin{enumerate}
\item
Some easy developments: count the number of Bohr-Sommerfeld fibers and compare to the Verlinde formula (the dimension of conformal blocks, see \cite{tensor} or \cite{bhmv}). Explain how the functions $h_i$ are quantized and act naturally on the quantization. This also provides a quantization of the Dehn twists acting on $\boM(\Sigma)$ with an explicit spectrum.
\item
Given an other pants decomposition, we have a new construction of the quantization which can be compared to the previous one. Suppose that the fibrations are transverse to each other. Then the half-forms sections can be intersected and give a pairing between the two quantizations called Blattner-Kostant-Sternberg pairing (BKS pairing, see \cite{gs,weinstein}). It is not clear wether this pairing gives a unitary isomorphism between the two quantizations. If so, it would give a satisfactory description of the quantization of $\boM(\Sigma)$ where all trace functions would be quantizable, and the action of the mapping class group of $\Sigma$ would extend to the quantization.
\item
A 3-manifold $M$ bounding $\Sigma$ produces a Lagrangian immersion $\boM(M)\to\boM(\Sigma)$ with a flat section $CS$ of the bundle $\boL$ and a volume form $T$ on $\boM(M)$. If we manage to find a well-defined square root of this form, we obtain a semi-classical state associated to $M$. Using the BKS-pairing, this states may be viewed as a vector belonging to any quantization, see \cite{jw3}.
\item
All these data should sit into a Topological Quantum Field Theory (TQFT), that is may have functorial properties with respect to the gluing of 3-manifolds along their boundaries. Moreover, we expect that this TQFT appears as the semi-classical approximation of a family of TQFTs indexed by the integer $K$ called level. These TQFTs may be constructed either by geometric quantization (with complex polarization, see \cite{witten,physics}) or with link polynomials and quantum groups (see \cite{rt,bhmv}).
\end{enumerate}

\end{document}